\documentclass{amsart}
\usepackage{amsmath, amsthm, amssymb}
\usepackage{verbatim}
\usepackage{tikz}
\usetikzlibrary{matrix}

\usepackage[mathscr]{eucal} 
\usepackage{mathrsfs} 
\usepackage{cancel}

\usepackage{enumitem, hyperref}
\makeatletter
\def\namedlabel#1#2{\begingroup
    #2%
    \def\@currentlabel{#2}%
    \phantomsection\label{#1}\endgroup
}

\renewcommand{\d}{\partial}

\newcommand{\wt}[1]{\widetilde{#1}}
\newcommand{\lc}{\left<}
\newcommand{\rc}{\right>}
\newcommand{\cw}[1]{\check{#1}}

\newcommand{\eps}{\epsilon}
\newcommand{\veps}{\varepsilon}
\newcommand{\vphi}{\varphi}

\newcommand{\al}{\alpha} 
\newcommand{\ze}{\zeta}

\newcommand{\de}{\delta}
\newcommand{\la}{\lambda}
\newcommand{\om}{\omega}
\newcommand{\te}{\theta}

\newcommand{\vka}{\varkappa}
\newcommand{\ka}{\kappa}

\newcommand{\Om}{\Omega}
\newcommand{\De}{\Delta}

\newcommand{\cP}{\mathcal{P}}

\newcommand{\cH}{\mathcal{H}}
\newcommand{\cL}{\mathcal{L}}

\newcommand{\cO}{\mathcal{O}}

\newcommand{\bB}{\mathbb{B}}
\newcommand{\bK}{\mathbb{K}}
\newcommand{\bP}{\mathbb{P}}
\newcommand{\bR}{\mathbb{R}}

\newcommand{\bC}{\mathbb{C}}

\newcommand{\bN}{\mathbb{N}}
\newcommand{\vpi}{\varpi}

\newcommand{\wh}[1]{\widehat{#1}}
\newcommand{\ov}[1]{\overline{#1}}

\newcommand{\uU}[1]{{\rm U}#1}

\newcommand{\cali}[1]{\mathscr{#1}}

\newcommand{\Cc}{\cali{C}}

\newtheorem{thm}{Theorem}
\newtheorem{prop}[thm]{Proposition}
\newtheorem{lem}[thm]{Lemma}
\newtheorem{cor}[thm]{Corollary}

\theoremstyle{definition}

\newtheorem{defn}[thm]{Definition}
\newtheorem{remark}[thm]{Remark}

\newtheorem{expl}[thm]{Example}
\newtheorem*{rmkx}{Remark}

\newtheorem{ques}[thm]{Question}

\numberwithin{thm}{section}
\numberwithin{equation}{section}

\renewcommand{\[}{\begin{equation}}
\renewcommand{\]}{\end{equation}}

\title[Regularity of extremal functions on compact K\"ahler manifolds]{Regularity of the Siciak-Zaharjuta extremal function on compact K\"ahler manifolds}

\author{Ngoc Cuong Nguyen} 
\address{Department of Mathematical Sciences, KAIST, 291 Daehak-ro, Yuseong-gu, Daejeon 34141, South Korea}
\email{cuongnn@kaist.ac.kr}

\begin{document} 


\begin{abstract} We prove that the regularity of the extremal function of a compact subset of a compact K\"ahler manifold is a local property, and that the continuity and H\"older continuity are equivalent to classical notions of the local $L$-regularity and the locally H\"older continuous property in pluripotential theory. As a consequence we give an effective characterization of the $(\Cc^\al, \Cc^{\al'})$-regularity of compact sets, the notion introduced by Dinh, Ma and Nguyen. Using this criterion all compact fat subanalytic sets in $\bR^n$ are shown to be regular in this sense.
\end{abstract}

\keywords{Siciak-Zaharjuta extremal function, K\"ahler manifolds, $L$-regularity, H\"older continuous, subanalytic sets}

\maketitle

\section{Introduction}
{\em Background.}
Let $\cL$ denote the Lelong class of plurisubharmonic  functions on $\bC^n$ with logarithmic growth at infinity, i.e.,
$$	\cL = \left\{v \in PSH(\bC^n): u(z) - \frac{1}{2}\log (1+ |z|^2) \leq c_u \right\}.
$$
The classical extremal function for a bounded subset $E$  in $\bC^n$  is defined by
\[\label{eq:SZ-classic}
	L_{E}(z) = \sup\{ v(z) : v \in \cL, \, v \leq 0 \text{ on } E\}.
\]
This function was introduced by Siciak \cite{Si62,Siciak81} using the polynomials and characterized by Zaharjuta \cite{Za76} via plurisubharmonic functions. 
Since its introduction the extremal function has found numerous applications in pluripotential theory, approximation theory in $\bC^n$ and other areas. It often turns out that geometrical properties of the compact subset $E$ can be read from the regularity of its extremal function $L_E$. It is well-known that $E$ is non-pluripolar if and only if its upper semicontinuous regularization $L_E^*$ is bounded on $E$. 

The continuity of the extremal function is a classical topic. 
Let $K \subset \bC^n$ be a compact subset and $a\in K$. We say that
\begin{itemize}
\item[($1^o$)]  $K$  is  {\em $L$-regular at $a$} if $L_K$ is continuous at $a$;
\item[($2^o$)]
$K$ is  {\em $L$-regular} if $L_K$ is continuous at every point of $K$.
\end{itemize}
Notice that in ($2^o$) if $L_K$ is continuous at every point of $K$, then $L_K$ is continuous in $\bC^n$ by \cite{Siciak81} and \cite{Za76}. 
Its local version in \cite{Siciak81}  will be important for us. The compact set
$K$ is said to be 
\begin{itemize}
\item
[($3^o$)]  {\em locally $L$-regular at $a$} if $L_{K\cap B(a,r)}$ is continuous at $a$ for every closed ball $B(a,r)$ of radius $r>0$;
\item
[($4^o$)] {\em locally $L$-regular} if it is locally $L$-regular at every point of $K$.
\end{itemize}
This notion is stronger than the $L$-regularity as shown in \cite{Sa81} and  in \cite{Ce82}. It closely relates to the weighed extremal functions. Given 
a continuous function $\phi$ on $K$, the function $L_{K,\phi}$ is defined by
\[\label{eq:SZ-weighted}
	L_{K,\phi}(z) = \sup\{v(z): v \in \cL, v \leq \phi \text{ on } K\}.
\]
Siciak proved in \cite[Propositions~2.3,~2.16]{Siciak81} that the continuity of $L_{K,\phi}$ is equivalent to its continuity on $K$; and  if $K$ is locally $L$-regular, then $L_{K,\phi}$ is continuous. Conversely, Dieu showed in \cite{Dieu} that  if $L_{K,\psi}$ is continuous for every continuous weight $\psi$, then $K$ is locally $L$-regular. In Section~\ref{sec:local-regularity} we collected many well-known examples of the locally $L$-regular sets from \cite{Pl79, Pl84, Pl01} and \cite{Siciak97}. In particular, a compact domain whose boundary is $C^1$-smooth is locally $L$-regular.
Furthermore, we refer to the reader to the textbook by Klimek \cite[Chapter 5]{Kl91} for the detailed discussion and results up to 1990, and the surveys of Levenberg \cite{Le06} and Ple\'sniak \cite{Pl03} for more recent results. 

Next, the H\"older continuity of the extremal function has been studied extensively in relation to the Markov inequalities \cite{BC14}, \cite{BCEKNP17}, \cite{BiE16, BiE19} and \cite{Pl98}.
Let $K \subset \bC^n$ be a compact subset and $a\in K$. We say that
\begin{itemize}
\item[($5^o$)]
$K$ has {\em $\mu$-H\"older continuity property} ({\rm $\mu$-HCP} for short) {\em at $a$} if there exist constants $C>0$ and $\mu>0$ such that
$$
	L_K(z) \leq C \de^\mu, \quad |z-a| \leq \de\leq 1;
$$
\item
[($6^o$)] $K$ has {\rm $\mu$-HCP} if there exist constants $C>0$ and $\mu>0$ such that
$$
	L_K (z) \leq C \de^\mu, \quad {\rm dist}(z, K) \leq \de\leq 1.
$$
\end{itemize}
Notice in  ($6^o$) that if $L_K$ has {\rm $\mu$-HCP} on $K$, then $L_K$ has $\mu$-H\"older continuous on $\bC^n$ by \cite[Proposition~3.5]{Siciak97}. 
It was observed in \cite[Remark~3.7]{Siciak97} that in ($5^o$) and ($6^o$) the H\"older exponent satisfies $0<\mu \leq 1$. Moreover, in these two items it is sufficient to show the inequality  for $0<\de \leq \de_0$ with some constant $0<\de_0 \leq 1$.
The typical compact sets having $\mu$-HCP are convex sets with non-void interior either in $\bR^n\equiv \bR^n + i\cdot 0$, or in $\bC^n$.  A compact domain whose boundary is Lipschitz continuous has also $\mu$-HCP. These examples were contained in \cite{Siciak85}. Next, Paw\l ucki and Ple\'{s}niak \cite{PP86} showed that a very large class of uniformly polynomially cuspidal (UPC) sets have this property. The boundary of such sets may have cusp.  More recently, Sadullaev and Zeriahi \cite{SZ16} proved this property for a compact generic real submanifold of $\bC^n$. 

{\em Results.} Let $(X, \om)$ be a compact K\"ahler manifold of dimension $n$.
Let $PSH(X,\om)$ denote the space of $\om$-plurisubharmonic ($\om$-psh) functions, i.e.,
\[\notag
	PSH(X,\om) = \{v \in L^1(X): \om + dd^c v \geq 0\}.
\]
Note that if $c_2>c_1>0$, then $v\in PSH(X, c_1\om)$ implies $v\in PSH(X,c_2\om)$. If $\phi$ is quasi-psh, then $dd^c\phi+A\om\geq 0$ for some $A>0$. 
Our goal is to study the regularity (continuity and H\"older continuity) of 
the extremal functions associated to compact sets on the global setting. 
The Siciak-Zaharjuta extremal function of a Borel subset $E$ in $X$ is defined by
\[\label{eq:SZ-ext}
	V_E (z) = \sup\{ v(z): v \in PSH(X,\om) : v \leq 0 \quad\text{on } E\}.
\]
This function was introduced and studied early by Guedj and Zeriahi \cite[Chapter 9.4]{GZ17} in analogue with the one in $\bC^n$. For a real-valued continuous function $\phi$ on $X$,  the weighted extremal function is given by
\[\label{eq:w-SZ-ext}
	V_{E,\phi} (z) = \sup\{ v(z): v\in PSH(X,\om), \; v_{|_E} \leq \phi\}.
\]
(It is enough to assume $\phi$ is defined on $E$, however we are mainly interested in a compact set $E$ from which we can extend $\phi$ as a continuous function on the whole manifold.)
Berman, Boucksom and Witt-Nystrom  \cite{BBW11} showed that for many applications the continuity of $V_{E,\phi}$ is important, e.g., the Bernstein-Markov measure with respect to a plurisubharmonic weight $\phi$. 
Recently, Sadullaev \cite{Sa16} proved the equivalence between the global and local continuity of the extremal function for  $X=\bP^n$ with $\om= \om_{FS}$ the Fubini-Study metric.
Our first result extends this equivalence on any compact K\"ahler manifold.

\begin{thm}\label{thm:characterization-c} Let $K \subset X$ be a non-pluripolar compact subset and  $a\in K$. Let $B(a,r)\subset X$ denote a closed coordinate ball centered at $a$ with small radius $r>0$. 
\begin{itemize}
\item
[(a)] $V_K$ is continuous at $a$ if and only if  $V_{K \cap B(a,r)}$ is continuous at $a$.
\item
[(b)] Let $\phi\in C^0(X,\bR)$. If $V_{K}$ is continuous,  then $V_{K,\phi}$ is continuous.
\item
[(c)] $V_K$ is continuous if and only if $K$ is locally $L$-regular.
\end{itemize}
\end{thm}

Note that the sufficient condition in (a)  follows easily from the monotonicity. Unlike  in $\bC^n$ setting above the necessary condition holds on compact manifolds. This item extends the one in \cite{Sa16} whose proof relied on the facts that $\bP^n$ admits a large coordinate chart $\bC^n$ and on the relation between the extremal functions in $\bC^n$ and $\bP^n$. Here we  improve his proof so that it works for a general compact K\"ahler manifold. The item (b) is a consequence of (a), where  it is often the case that the modulus of continuity of the weighted extremal function is weaker than the unweighted one. The local $L$-regularity in (c) is understood in the sense of $(4^o)$, because  we may assume that the compact set is contained in a holomorphic coordinate chart,  which is biholomorphic to a ball in $\bC^n$, thanks to the characterization in (a).

Next, we wish to study the H\"older regularity of the global extremal functions.
When $K= X$, the weighted extremal function \eqref{eq:w-SZ-ext} becomes the envelope. It  has been well-understood. For example,  it is proved by Tosatti \cite{To18} that if $\phi$ is  smooth then its envelope has the optimal $C^{1,1}$-regularity (see also Berman \cite{B19}). More generally, if $\phi$ is $C^{0,\al}(X)$ for $0\leq \al \leq 1$, then the same regularity of the envelope  is proved in \cite{LPT} (see also \cite{ChZ19}, \cite{GLZ19}).
However, the problem becomes very different for a general compact subset, which is important for applications. This is our main focus. It follows from \cite{DMN} that the Siciak-Zaharjuta extremal function associated to a $C^2$-smooth compact domain is Lipschitz continuous. Later, this property is obtained for the extremal function associated to a smooth real generic submanifold in \cite{Vu18}.

We say that a compact subset $K$ has {\em a uniform density in capacity} if there exist constants $q>0$ and $\vka>0$ such that for every $a\in K$, 
\[\label{eq:cap-density} 
\inf_{0<r <1} \frac{cap(K \cap B(a,r))}{r^q} \geq \vka,
\]
where $cap(\bullet)$ is the Bedford-Taylor capacity on $X$. This property  holds for all compact sets  (Lemma~\ref{lem:cap-density}) that satisfy the local H\"older regularity version of $(5^o)$ and $(6^o)$ (Definition~\ref{defn:local-mu-hcp}).
 Furthermore, it is a local property and it holds for most of natural sets as we will see in Sections~\ref{sec:HCP}, \ref{sec:UPC}, \ref{sec:totally-real}. 
Our next result provides a general criterion to study the H\"older regularity.
\begin{thm}\label{thm:characterization-hcp} Let $K \subset X$ be a non-pluripolar compact subset and  $a\in K$. Let $B(a,r)\subset X$ denote a closed coordinate ball centered at $a$ with small radius $r>0$. Assume $\phi$ is a H\"older continuous function on $X$.
\begin{itemize}
\item
[(a)] $V_K$ is $\mu$-H\"older continuous  at $a$ if and only if $V_{K \cap B(a,r)}$ is $\mu$-H\"older continuous at $a$.
\item
[(b)] Assume $K$ has a uniform density in capacity of \eqref{eq:cap-density}.   If $V_K$ is H\"older continuous, then $V_{K,\phi}$ is  H\"older continuous.
\item
[(c)]  $V_K$ is H\"older continuous and $K$ has a uniform density in capacity if and only if  $K$ has locally H\"older continuous property of order $q$.
\end{itemize}
\end{thm}

The item (a) is a useful result. This combined with Demailly's regularization theorem allows us to prove the general sufficient condition in (b). Notice also that as in Theorem~\ref{thm:characterization-c}-(b) the H\"older exponent in (b) of the weighted extremal function is often smaller than the one of the unweighted extremal function.
The item (c) is a new and effective criterion (the locally H\"older continuous property will be given in Definition~\ref{defn:local-mu-hcp}). It can be applied to the previous works  \cite{DMN}, \cite{SZ16} and \cite{Vu18}, where the H\"older continuity of the (weighted) extremal functions were proved directly. The new input here is  a precise estimate on the H\"older norm/coefficient, which is the growth of sup-norm of $V_{E\cap B(a,r)}$ like $r^{-q}$ for $q>0$ and $r>0$ small in (c). 

To state the applications of the characterization of the H\"older continuity, we recall the notion  is introduced by Dinh, Ma and Nguyen \cite{DMN}. 

\begin{defn} \label{defn:DMN-regular}
 A non-pluripolar compact subset  $K\subset X$   is said to be  {\em $(\Cc^\al, \Cc^{\al'})$-regular}, where $0 < \al, \al' \leq 1$,  if for every $\al$-H\"older continuous weight $\phi$, its weighted extremal function $V_{K,\phi}$ is $\al'$-H\"older continuous.
\end{defn}
Obviously, if $K$ is  $(\Cc^\al, \Cc^{\al'})$-regular, then  $V_{K}$ is necessary H\"older continuous for the weight $\phi\equiv 0$.   As shown in \cite{DMN} one can  obtain the speed of convergence of Fekete's measures associated to  such a compact subset. There are more applications of such regular sets recently found in \cite{MV22}.
Thanks to Theorem~\ref{thm:characterization-hcp}-(b) and (c)  we are able to prove the $(\Cc^\al, \Cc^{\al'})$-regularity for many new examples in Section~\ref{sec:regularity}.  Among them the following subclass is contained in the class of UPC sets. This was somehow conjectured by Zeriahi \cite[page 562]{DMN}.

\begin{thm}\label{thm:UPC-intro}  A compact fat subanalytic subset in $\bR^n$ is $(\Cc^\al, \Cc^{\al'})$-regular.
\end{thm}

Here a compact subanalytic subset $E$ of $\bR^n$ is fat if $E = \ov{{\rm int\, } E}$, and in general, these sets admit a cusp singularity. These subsets are fundamental objects in real algebraic geometry \cite{BM88}.  By invoking \cite{DMN} we obtain the speed of convergence of sequence of Fekete's measures associated to a compact fat subanalytic subset in $\bR^n$ (Theorem~\ref{thm:equidistribution}).

\bigskip
{\em Organization.} In Section~\ref{sec:c} we first provide basic properties of the extremal functions on a compact K\"ahler manifold. Then, we prove the characterizations on the continuity in Theorem~\ref{thm:characterization-c}.  In Section~\ref{sec:hcp}, we start with the study of the local uniform density in capacity \eqref{eq:cap-comparison} and recall the fundament approximation theorem of Demailly. The crucial result is Lemma~\ref{lem:property}. Next, we complete the proof of characterizations of H\"older continuity (Theorem~\ref{thm:characterization-hcp}). Section~\ref{sec:regularity} is devoted to study many examples of local $L$-regularity and local $\mu$-HCP compact sets in $\bC^n$. The local $\mu$-HCP of compact fat subanalytic sets in Theorem~\ref{thm:UPC-intro} is a consequence of Theorem~\ref{thm:characterization-hcp}-(c) and Corollary~\ref{cor:fat-analytic}. Section~\ref{sec:app} contains some applications related to the speed of convergence for sequences of associated measures with Fekete points of  fat subanalytic sets  in $\bR^n$. Lastly, Section~6 consists of an appendix which justifies the reduction to working on compact subsets of $\bC^n$ in $\bP^n$ for the H\"older continuous results.

\bigskip
{\em Acknowledgement.} I would like to thank W. Ple\'sniak for providing me many copies of his papers including the one \cite{PP86}. These have been very valuable resources for many questions investigated here.
I would also like to thank  S.\,Ko\l odziej for reading the manuscripts and giving many useful comments. I am benefited from a careful reading of Hyunsoo Ahn. The author is very grateful to the referee who has read the paper very carefully and given many valuable suggestions. Thanks to the suggestions the exposition of the paper is significantly improved. The author is partially supported by  the National Research Foundation of Korea (NRF) grant  no. 2021R1F1A1048185.

\section{Continuity of the extremal functions}\label{sec:c}

\subsection{Basic properties of the global extremal functions}
In this section we recall basic results related to the  extremal functions and weighted extremal functions on a compact K\"ahler manifold $(X,\om)$ of dimension $n$. These results are the analogues of the classical results in \cite{Siciak81, Siciak82}, \cite{Kl91} on the Siciak-Zaharjuta extremal function on $\bC^n$. The detailed proofs are contained in \cite{GZ17}.  
The first one is \cite[Theorem~9.17]{GZ17}.
\begin{prop}\label{prop:basic} Let $E \subset X$ be a Borel set. Then,
\begin{itemize}
\item
[(a)] $E$ is pluripolar $\Leftrightarrow$ $\sup_X V_E^* \equiv +\infty$ $\Leftrightarrow$ $V_E^*\equiv +\infty.$
\item
[(b)] If $E$ is not pluripolar, then $V_E^* \in PSH(X,\om)$. Moreover, $V_E^* \equiv 0$ in the interior of $E$.
\end{itemize}
\end{prop}

Next, we have \cite[Proposition~9.19]{GZ17} for the basic properties. 
\begin{prop}\label{prop:elementary} \mbox{}
\begin{enumerate}
\item[(a)]
If $E\subset F$, then $V_E \leq V_F$.
\item[(b)]
If $E$ is an open subset, then $V_E = V^*_{E}$.
\item[(c)]
If $P \subset X$ is  pluripolar, then $V_{E\cup P}^* = V_E^*$.
\item[(d)]
Let $\{E_j\}$ be an increasing sequence of  sets in $X$ and $E := \cup E_j$, then
$\lim_{j\to \infty}V_{E_j}^* = V_E^*.$
\item[(e)]
Let $\{K_j\}$ be a decreasing sequence of compact sets in $X$ and $K:= \cap K_j$, then $\lim_{j\to \infty} V_{K_j} = V_K$. Furthermore, $\lim_{j\to \infty} V_{K_j}^*  = V_K^*$ a.e.
\end{enumerate}
\end{prop}

We will frequently need the following result for compact K\"ahler manifolds. 

\begin{lem} \label{lem:lower-semicontinuity} If $K\subset X$ is a compact subset, then $V_K$ is lower semi-continuous. 
\end{lem}

\begin{proof} Let $v \in PSH(X, \om)$ and $v\leq 0$ on $K$. By Demailly's regularization theorem \cite[Proposition~3.8]{De94} (see also \cite{BK07}) there exists a sequence of smooth function $v_j \in PSH(X,\om) \cap C^\infty(X)$ decreasing to $v$. Fix $\de>0$. By Hartogs' lemma we have for $j\geq j_0$,
$$
	v_j - \de \leq 0 \quad \text{on } K.
$$
Consequently, $V_K$ is supremum of a family of continuous functions, and the result follows.
\end{proof}

This gives a simple criterion to check the continuity of the extremal functions.

\begin{cor}\label{cor:semicontinuity} If  $K$ is a compact subset and $V_K^* \equiv 0$ on $K$, then $V_K$ is continuous on $X$.
\end{cor}


\begin{remark}\label{rmk:regularity-c} The continuity of $V_K$ is a property of the set itself, i.e., it is independent of reference K\"ahler metric $\om$.  Indeed, assume $\om'$ is another K\"ahler metric and $K$ is regular with respect to $\om$. There exists  $A>0$ such that $\om' \leq A \om$. It follows that 
$
	 V_{\om';K}^* \leq A V_{\om;K}^*.
$ 
Hence, the continuity of $V_{\om';K}$ follows from Corollary~\ref{cor:semicontinuity}.
Similarly, the independency holds also for the H\"older continuity of $V_K$ as in Lemma~\ref{lem:property} below, where the H\"older exponent is  independent of the reference metric too. For this reason if there is no confusion then we only write $V_K$ for the extremal function with respect to a fixed K\"ahler metric.
\end{remark}

The "zero-one" relative extremal function for a Borel set $E\subset X$ is given by
\[\label{eq:r-ext}
	h_E (z) := \sup\left\{ v(z): v\in PSH(X,\om),\; v\leq 1, \quad v_{|_E}\leq 0\right\}.
\]
We know that $0\leq h_E^* \leq 1$ and $h_E^* \in PSH(X,\om)$. Furthermore,  $E$ is pluripolar if and only if $h_E^* \equiv 1$ (see \cite[page 620]{GZ05}).  

For non-pluripolar compact set $K$, let us denote  $M_K := \sup_{X}V_K^*<+\infty$. Then, it is easy to see that 
\[\label{eq:zero-one-function}
	V_K^* \leq M_K h_K^*.
\]
We will need an estimate of $M_K$ in term of capacity.  Recall that for a Borel set $E\subset X$, the Bedford-Taylor capacity \cite{Ko05}  is defined by
\[\label{eq:BT-cap}
	cap(E) = \sup\left\{ \int_E (\om + dd^cv)^n: v\in PSH(X,\om), \, -1 \leq v \leq 0 \right\}.
\]
(Here $cap = cap_\om$, we omit  the subscript $\om$ if it is already fixed and there is no confusion arises). Normalizing the metric $\om$ so that 
$$
	\int_X \om^n =1.
$$
Thus, for every compact subset $K\subset X$,
$cap (K) \leq 1.$
By \cite[Lemmas~12.2, 12.3]{GZ17} we know that there exists a uniform constant $A$ such that for every compact subset $K$, 
\[ \label{eq:cap-comparison} 
	\frac{1}{[cap(K)]^\frac{1}{n}} \leq \sup_X V_K \leq \frac{A}{cap(K)}.
\]

Next, we consider the weighted extremal function 
$$
	V_{K,\phi} (z) = \sup\{ v(z): v\in PSH(X,\om), \; v_{|_K} \leq \phi\},
$$
where $\phi$ is a real-valued continuous function on $X$. 
If $E \subset F$ be non-pluripolar compact subsets and $\phi \leq \psi$ are continuous functions, then we have the following monotonicity 
\[ \label{eq:monotonicity}
	V_{F, \phi} \leq V_{E,\phi} \leq V_{E, \psi}.
\]
Another useful property is 
\begin{lem}\label{lem:weight-vs-unweight} Let $K\subset X$ be a non-pluripolar compact subset and $\phi$ a continuous function on $X$.
\begin{itemize}
\item
[(a)] $V_K^* + \inf_K \phi \leq V_{K,\phi}^* \leq V_K^* + \sup_K \phi.$
\item
[(b)] Let $\te = \om + dd^c \phi$ and
$
	V_{\te;K} = \sup\{v \in PSH(X, \te): v_{|_K} \leq 0\}.
$
Then,
$$
	V_{K, \phi} = V_{\te;K} + \phi,
$$
\item
[(c)] $V_{K,\phi}$ is continuous if and only if $V_{K,\phi}^* \leq \phi$ on $K$.

\end{itemize}
\end{lem}
\begin{proof} The items (a) and (b) follow immediately from the definitions of $V_K$ and $V_{K,\phi}$. To prove (c) let us write $V:= V_{K,\phi}$. Since $K$ is compact, Demailly's regularization theorem \cite[Proposition~3.8]{De94} (see also \cite{BK07}) implies also that $V$ is lower semi-continuous. The conclusion follows easily from the definition as $V=V^*$. 
\end{proof}

\begin{remark}
Notice that there are more general relative extremal functions studied by Di Nezza and Lu \cite{DiL15} which is the global version of the weighted extremal function introduced by Bedford \cite{Be87} and used for the subextension problem by Cegrell, Ko\l odziej and Zeriahi \cite{CKZ}. The relative weighted extremal  functions in $\bC^n$ was also studied in Alan' thesis \cite{Al-thesis}.
\end{remark}

In the special case $X = \bP^n$ with $\om= \om_{FS}$ the Fubini-Study metric,  there is  a 1-1 correspondence between $PSH(\bP^n, \om)$  and $\cL(\bC^n)$ (see e.g. \cite{GZ17}). 
Let us denote $\rho = \frac{1}{2}\log (1+ |z|^2)$ the local potential of  $\om$ on $\bC^n \subset \bP^n$, i.e.,
$$
	\om = dd^c \rho \quad\text{on } \bC^n.
$$ 
Recall that for a compact set  $E\subset \bC^n$, 
\[ \notag
	L_{E,\rho} (z)= \sup \left\{f(z): f \in \cL,\; f_{|_E} \leq \rho\right\}.
\]
Hence, 
\[\label{eq:identity-LV}
	L_{E,\rho} (z) =V_{E} (z) + \rho (z), \quad z\in \bC^n.
\]
Similar to Lemma~\ref{lem:weight-vs-unweight}-(a) we have
\[\label{eq:weight-vs-unweight}
	L_E^* + \inf_E \rho \leq L_{E, \rho}^* \leq L_E^* +\sup_E\rho.
\]

\subsection{Characterizations of the continuity}
\begin{proof}[Proof of Theorem~\ref{thm:characterization-c}]
(a)
Since $V_F^* \leq V^*_E$ for Borel sets $E \subset F$, the  sufficient condition is obvious. Next, we prove the necessary condition.
Assume $V_K^*(a) = 0$, we  need to show that  $V_{K\cap B}^*(a)=0$, where $ B := B(a,r)$ is a closed coordinate ball centered at $a$ with small radius $r>0$. Indeed, let us 
consider the positive relative extremal function defined in \eqref{eq:r-ext},
$$
	h_{K \cap  B}(z) = \sup\left\{ v(z): v \in PSH(X,\om),\; v_{|_{K\cap  B}} \leq 0, v \leq 1\right\}.
$$
Then, $h^*_{K\cap B} \in PSH(X,\om)$ with $0\leq h_{K\cap B}^* \leq 1$.

Let $0 \leq \chi \leq  1$ be a smooth function on $X$ such that $\chi(a)=0$ and $\chi \equiv 1$ on $X\setminus B(a,r)$. We have
\[\label{eq:cutoff-fct}
	\|\chi\|_{C^1(X)} \leq c_1 /r, \quad \|\chi\|_{C^2(X)} \leq c_2 / r^2,
\]
where $c_1, c_2$ are uniform constants independent of $a$ and $r$, the norms $\|\cdot\|_{C^1(X)}$ and $\|\cdot\|_{C^2(X)}$ are the usual $C^1$ and $C^2$ norms on $X$.
Hence, there exists $0< \veps \leq 1/2$, which is a small multiple of $r^2$, such that $-\veps \chi$ belongs to $PSH(X, \om/2)$.
Given $u \in PSH(X,\om)$ satisfying $u \leq 1$ and $u\leq 0$ on $K \cap  B$, we define 
$$
	\vphi(z) = \veps u(z) - \veps \chi(z).
$$
Then, $\vphi(z) \leq 0$ on $K$ and it is $\om$-psh. It follows from definition of $V_K$ that
$	\vphi(z) \leq V_{K}(z).
$
Taking supremum over all such $u$ we get 
\[\label{eq:basic-ineq}
	\veps  h^*_{K\cap  B}(z) - \veps \chi(z)  \leq V_{K}^*(z).
\]
Since $V_K^*(a) =0= \chi(a)$, we have by \eqref{eq:zero-one-function} that $h_{K\cap  B}^*(a) =0$. In particular,  $K \cap B$ is non-pluripolar. Denote  $M := \sup_X V_{K\cap  B}^*$. It follows from \eqref{eq:cap-comparison} that 
 $$ 1\leq M  \leq \frac{A}{cap(K \cap B)} < +\infty$$
for a uniform constant $A$. Using \eqref{eq:zero-one-function} again we get $V_{K\cap B}^* \leq M h^*_{K\cap B}$. So, $V_{K\cap B}^*(a) =0$ and the proof  of (a) follows.

(b) Assume $V_K$ is continuous and we write $V:= V_{K,\phi}$. Thanks to  Lemma~\ref{lem:weight-vs-unweight}-(c) we need to show that $V^*\leq \phi$ on $K$. In fact, let $a\in K$ and $\veps'>0$. We can choose $r>0$ so small that $\phi(x) \leq \phi(a) + \veps'$ in $B(a,r)$. By monotonicity \eqref{eq:monotonicity},
 we have
$$
	V \leq V_{K\cap B(a,r), \phi(a)+ \veps'} = \phi(a)+\veps' + V_{K\cap B(a,r)} \quad\text{on } X.
$$
Since $V_K$ is continuous, $V_{K\cap B(a,r)}$ is continuous at $a$ by (a). Hence $V^*(a) \leq \phi(a)+\veps'$. Let $\veps'\to 0$, we get $V^* (a)\leq \phi (a)$. The proof is completed.

(c) It follows from (a) that we may assume that $K$ is contained in a closed holomorphic chart  in $X$ which is biholomorphic to the closed unit ball in $\bC^n$. Since the continuity of the extremal function is a property of the set and it is invariant under biholomorphic maps, without loss of generality, we assume $K\subset\bC^n \subset \bP^n$ equipped with the Fubini-Study metric.

Assume that $K$ is locally $L$-regular at $a\in K$.
The argument in the proof of (b) applied for $L_K$ shows that 
the local $L$-regularity implies that $L_{K,\rho}$ is continuous on $\bC^n$ which was first proved in  \cite[Proposition~2.16]{Siciak81}. Hence,  $V_K$ is continuous by the identity \eqref{eq:identity-LV} and Corollary~\ref{cor:semicontinuity}.

Conversely, assume $V_K$ is continuous. Let $a\in K$ and $B_r:= B(a,r) \subset \bC^n$ be a closed ball of radius $r>0$. It follows from (a) that $V_{K\cap B_r}$ is continuous at $a$. The identity \eqref{eq:identity-LV} implies that $L_{K \cap  B_r,\rho}$ is continuous at $a$ for every $r>0$. Hence, by definition of the weighted extremal function in \eqref{eq:SZ-weighted},
$$
	L_{K\cap B_r,\rho} \leq \rho \quad\text{on } E.
$$
 Fix a ball $B:=B(a,r_0)$. Using \eqref{eq:weight-vs-unweight} we have, for every $0<r <r_0$,
$$
	 L_{K\cap  B}^*(a) \leq L_{K\cap  B_r}^*(a) \leq  L_{K\cap  B_r,\rho}^*(a) - \inf_{K\cap  B_r} \rho \leq \rho(a) -\inf_{K\cap  B_r} \rho.
$$
Since $\rho$ is continuous on $\bC^n$, letting $r\to 0$ we conclude that $L_{K\cap  B}^*(a) =0$. This means that $L_{K\cap B}$ is continuous at $a$. Since $a$ is arbitrary the proof of (c) is completed.
\end{proof}

It would be also interesting to know the equivalence between global and local regularity of weighted extremal functions. Let $\phi \in C^0(X,\bR)$ and let $B(a,r)\subset X$ be a closed coordinate ball centered at $a$ with small radius $r>0$. It is easy to see that $V_{K\cap B(a,r),\phi}$ is continuous at $a$, then $V_{K,\phi}$ is continuous at $a$. 

\begin{ques}\label{q:sa} Assume $V_{K,\phi}$ is continuous at $a$.   Is  $V_{K\cap B(a,r),\phi}$ is continuous at $a$?
\end{ques}

If $X = \bP^n$ with  the Fubini-Study metric $\om$, then for a bounded Borel set $E\subset \bC^n$,
$$
	L_{E, \phi+\rho} (z)= V_{E,\phi}(z) +  \rho (z), \quad z\in \bC^n.
$$
Sadullaev obtained the positive answer to Question~\ref{q:sa}  in \cite[Theorem~2.4]{Sa16} provided that  $\psi := \phi+\rho$ is  a $C^2$-smooth strictly psh  function in  $\cL(\bC^n)$. He also asked in \cite[Problem~2.8]{Sa16} whether it is enough to assume $\psi$ is strictly psh in a neighborhood of $K$. We have a positive answer for his question in the case  $\psi = \phi +\rho$, where $\phi$ is either $C^2$-smooth or continuous quasi-psh on $\bP^n$. This is also a partial answer for the above question.

\begin{lem}\label{lem:loc-w-c} Let $\phi \in C^0(X,\bR)$. Let $K\subset X$ be compact and $a\in K$. Assume $\phi$ is quasi-psh on $X$, i.e., 
$$\te= dd^c \phi + A\om \geq \om$$
on $X$ in the weak sense  for some constant $A>0$. Then, $V_{K,\phi}$ is continuous at $a$ if and only if $V_{K \cap B(a,r),\phi}$ is continuous at $a$.
\end{lem}

\begin{proof} Using the relation between the weighted and unweighted extremal functions in Lemma~\ref{lem:weight-vs-unweight}, for every Borel set $E$,
$$
	V_{E,\phi} = V_{\te; E} + \phi.
$$
Hence, $V_{E,\phi}^* = V_{\te;E}^* + \phi$. Notice that $\te$ may not be smooth, but it is strictly positive whose potential is continuous. Consequently, the Bedford-Taylor capacity $cap_\om$ and $cap_\te$  (see \eqref{eq:BT-cap} above) are equivalent to each other and the inequality \eqref{eq:zero-one-function} holds true. Furthermore, Lemma~\ref{lem:lower-semicontinuity}, Corollary~\ref{cor:semicontinuity} are valid for $PSH(X,\te)$. Hence, the argument in the proof in Theorem~\ref{thm:characterization-c}-(a) still holds in $PSH(X,\te)$. It shows that $V_{\te; K}$ is continuous at $a$ if and only if $V_{\te;K\cap B(a,r)}$ is continuous at $a$. This finished the proof.
\end{proof}

As a consequence we have the converse direction of Theorem~\ref{thm:characterization-c}-(b) for special weights.

\begin{cor}\label{cor:loc-w-c} Assume $\phi$ is continuous quasi-psh on $X$. If $V_{K,\phi}$ is continuous, then   $V_K$ is continuous.
\end{cor}

\begin{proof} By assumption there exists a constant $A>0$ such that $\te= A\om + dd^c \phi \geq \om$. Then, $$
	V_{K,\phi} = V_{\te; K} + \phi.
$$
It follows that the continuity of $V_{K,\phi}$ is equivalent to the one of $V_{\te;K}$. By  Lemma~\ref{lem:loc-w-c}, $V_{\te;K\cap B(a,r)}$ is continuous at $a\in K$ for every closed coordinate ball $B(a,r)$. Hence, using the proof of Theorem~\ref{thm:characterization-c}-(b), we get that  $V_{\te;K, -\phi}$ is continuous. Moreover, using one more time the relation
$$
	V_{\te;K, -\phi} = V_{A\om; K} + \phi,
$$
we get the continuity of $V_{A\om;K}$ and that of $V_{K}$.
\end{proof}

\section{H\"older continuity of the extremal functions}\label{sec:hcp}

\subsection{Local uniform density in capacity} Let $K\subset X$ be a compact subset in a compact K\"ahler manifold $(X,\om)$ of dimension $n$. 
First, we show that this property is equivalent to having the control of sup-norm of the extremal function on coordinate balls.  
Assume that $K$ has a uniform density in capacity \eqref{eq:cap-density}, i.e., there exist constants $q>0, \ka>0$ such that for every $a\in K$,
$$
\inf_{0<r <1} \frac{cap(K \cap B(a,r))}{r^q} \geq \vka.
$$
By the second inequality in \eqref{eq:cap-comparison} for the compact set $K\cap B(a,r)$
$$
	\sup_X V_{K \cap B(a,r)} \leq  \frac{A}{\vka \, r^q}, \quad  0< r <1.
$$
Conversely, if we have the control  $\sup_X V_{K \cap B(a,r)} \leq A / r^q$ for $0<r<1$, where $A, q>0$ are uniform constants, then $K$ has uniform density in capacity, i.e.,
$$
	\frac{cap(K \cap B(a,r))}{r^{q/n}} \geq \frac{1}{A^n}, \quad  0<r<1.
$$

Secondly, by definition the uniform density with capacity is a local property. We will see that  it can be verified by using the Bedford-Taylor capacity \cite{BT82} in a local coordinate. Without loss of generality we may assume that $K$ is contained in a coordinate unit ball $\Om$. Then, by considering its image in that chart, 
we may assume that  $$K  \subset  B(0,1/2) \subset \Om:= B(0,1) \subset \bC^n.$$ By equivalence between the global Bedford-Taylor capacity and the local one \cite[Eq. (6.2)]{Ko05}, $K$ has a uniform density in capacity if and only if 
\[\label{eq:cap-density-loc}
	\frac{cap'(K \cap  B(a,r), \Om)}{r^q} \geq \vka, \quad 0<r<1,
\]
where for a Borel subset $E\subset \Om$, 
$$
	cap'(E, \Om) = \sup\left\{ \int_E (dd^c u)^n : u \in PSH(\Om), -1 \leq u \leq 0\right\}.
$$
Then, by applying the comparison between two capacities \cite{AT84} (see also \cite[Theorem~2.7]{Ko05}), which is the local version of  \eqref{eq:cap-comparison},  the condition \eqref{eq:cap-density-loc} is equivalent to the existence of uniform constant $A, q'>0$ such that for every $a\in K$ and  $0<r<1$,
$$
	\sup_{\Om} L_{K\cap  B(a,r)} \leq \frac{A}{r^{q'}},
$$
where $L_{K\cap B(a,r)}$ is defined as in \eqref{eq:SZ-classic}.
More precisely, if $\sup_\Om L_{K \cap  B(a,r)} \leq A/r^q$, then 
\[\label{eq:cap-density-sufficient-loc}
	\frac{cap'(K\cap B(a,r),\Om)}{r^{nq}} \geq \frac{1}{A^n}.
\]
\begin{remark}  To verify the uniform density in capacity in \eqref{eq:cap-density} or \eqref{eq:cap-density-loc}  it is enough to have  a uniform $0<r_0 \leq 1$ satisfying
$$
	\inf_{0<r <r_0} \frac{cap(K \cap B(a,r))}{r^q} \geq \vka.
$$
Then, by monotonicity of the capacity we get the whole range $0<r<1$.
\end{remark}

As in the continuous case it will be shown in Theorem~\ref{thm:characterization-hcp}-(a) (see the appendix for the detail) that it is no loss of generality to consider compact subsets $K\subset \bC^n \subset \bP^n$ for the H\"older continuous case. In this setting,  the uniform density in capacity can be seen from the H\"older continuity of $L_K$. Recall that the Lelong class in $\bC^n$ is given by
\[\label{eq:L-class}\notag
	\cL = \left\{ f \in PSH(\bC^n):  f (z) - \rho(z) < c_f\right\},
\]
where $\rho = \frac{1}{2}\log (1+ |z|^2)$. For a  non-pluripolar compact subset $E$ in $\bC^n$,
$$
	L_E(z) := \sup \left\{f(z): f \in \cL,\;  f_{|_E} \leq 0\right\}.
$$
The modulus of continuity of $L_E$ at $a\in E$ is given by  
$$
	\vpi_E'(a, \de) = \sup_{|z-a| \leq \de} L_E(z)
$$
for $0< \de \leq 1$. Then, $L_E$ is continuous at $a$ if and only if 
$\lim_{\de\to 0} \vpi'_E(a,\de) =0$. 
Put
$$
	\vpi_E'(\de) = \sup\{ \vpi_E'(a,\de): a\in E\}
$$
which is the modulus of continuity of $L_E$ over $E$. 
It is a well-known fact due to B\l ocki \cite[Proposition~3.5]{Siciak97} that the modulus of continuity of $L_E$ on $E$ controls the modulus of continuity of $L_E$ on $\bC^n$. Namely,
\[\label{eq:local-modulus-of-continuity}
	|L_E(z) - L_E(w)| \leq \vpi'_E(|z-w|), \quad z,w\in \bC^n, \; |z-w|\leq 1.
\]
Inspired from the definition of local $L$-regularity in \cite{Siciak81}  we introduce 

\begin{defn}\label{defn:local-mu-hcp} Let $q\geq 0$ be an integer and  $0< \mu \leq 1$. Let $B(a,r)$ be a closed ball with center at $a$, of radius $r>0$. We say that  a compact subset $K\subset \bC^n$  has
\begin{itemize}
\item
[(a)]  locally $\mu$-H\"older continuity property (local {\rm $\mu$-HCP} for short) of order $q$ at $a\in K$ if there exist constants  $C>0$ and $0< r_0 \leq 1$  such that 
$$
	\vpi_{K\cap B(a,r)}' (a,\de) \leq \frac{C  \de^\mu}{r^q}, \quad 0<\de\leq 1,\, 0< r <r_0;
$$
\item
[(b)] local $\mu$-HCP of order $q$ if it has local $\mu$-HCP of order $q$ at every point $a\in K$ for the constants $C, r_0$ being  independent of $a$.
\end{itemize}
\end{defn}
Here it is important to know both the H\"older exponent and the H\"older coefficient. 
A weaker property is the HCP, introduced in $(5^o)$ and $(6^o)$, where  it only requires the H\"older continuity of $L_K$. 
The  definition implies an important property 
\begin{lem} \label{lem:cap-density} Let $K \subset \bC^n$ be a non-pluripolar compact subset. If $K$ has local {\rm $\mu$-HCP} of order $q$,  then it has a uniform density in capacity in the sense of \eqref{eq:cap-density-loc}.
\end{lem}

\begin{proof} 
Without loss of generality we may assume that $K \subset \Om:=B(0,1)$. 
By the assumption there exist uniform constants $q>0$ and $C, r_0>0$ such that for every $a\in K$, 
$$
	\vpi_{K \cap B(a,r)}' (a,\de) \leq \frac{C \de^\mu}{r^q}, \quad 0<\de \leq 1, \, 0<r<r_0.
$$
Let us fix $\de =1$. It follows that, for $0<r<r_0$,
$$
	\sup_{\Om} L_{K\cap B(a,r)} (z) \leq C/r^q.
$$
It follows from the inequality \eqref{eq:cap-density-sufficient-loc} that
$$
	\frac{cap'(K\cap B(a,r),\Om)}{r^{nq}} \geq \frac{1}{C}.
$$
This finished the proof.
\end{proof}
Noting that these local {\rm $\mu$-HCP} properties will be used in Section~\ref{sec:regularity}.

\subsection{Demailly's approximation theorem.} This is the main technical tool for proving the characterizations of the H\"older continuity. In the proof of the continuous case (Lemma~\ref{lem:lower-semicontinuity}) an elementary approximation sequence in \cite{BK07} was enough. However, in the H\"older continuous case we need to know precisely the negative part of complex Hessian matrices that is lost in the regularization sequence. To this end we will need the approximation sequence due to Demailly \cite{De94}.

Let $u$ be a bounded $\om$-psh function on $X$.
Considering $\Psi_{\delta }u$ the regularization of the $\omega$-psh function $u$ defined  by

\[\label{eq:phie}
\Psi_\delta u(z)=\frac{1}{\delta ^{2n}}\int_{\zeta\in T_{z}X}
u({\exp}_z(\zeta))\chi\Big(\frac{|\zeta|^2_{\omega }}{\delta ^2}\Big)\,d\la(\zeta),\ \delta>0;
\]
where $\zeta \mapsto {\exp}_z(\ze): T_zX \to X$ is the Riemannian  exponential map, $d\la$ denote the Lebesgue measure on the Hermitian space $(T_zX, \om(z))$  and the smoothing kernel 
 $\chi: \mathbb R_{+}\rightarrow\mathbb R_{+}$  is given by
$$\chi(t)=\begin{cases}\frac {\eta}{(1-t)^2}\exp(\frac 1{t-1})&\ {\rm if}\ 0\leq t\leq 1,\\0&\
{\rm if}\ t>1\end{cases}$$
 with a suitable constant $\eta$, such that
\[\notag
\int_{\mathbb C^n}\chi(\Vert z\Vert^2)\,dV_{2n}(z)=1
\]
($dV_{2n}$ being the Lebesgue measure in $\mathbb C^n$).
By the explanation in \cite[page 624]{DDGHKZ} the proof in \cite[Lemma~4.1]{KN19} holds for the exponential function $\exp_z$ in the place of its holomorphic part ${\exp}h_z$. The proof itself  was originated from \cite[Proposition~3.8]{De94}  and \cite[Lemma~1.12]{BD12}.  We restate the result here.

\begin{lem}\label{lem:kis} Let $u\in PSH(X,\om)\cap L^\infty(X)$.
 Define the Kiselman-Legendre transform with level $b>0$ by
 \begin{equation}\label{kisleg}
 U_{\delta, b}= \inf _{ t\in [0,\delta ]}(\Psi_{t }u + c_1t^2 + c_1t -b \log\frac{t}{\delta }),
 \end{equation}
Then, there exist positive constants $c_0, c_1$ depending only on $X,\om$ such that 
\begin{itemize}
\item[(i)]
$\Psi_{t }u+c_1t^2$ is increasing in $t$, 
\item
[(ii)] 
$
\omega+dd^c  U_{\delta,b }\geq -(c_0b+2c_1\delta)\,\omega.
$
\end{itemize}
\end{lem}
Notice that this function is different by a constant $c_1\de^2 + c_1\de$ compared with the one considered there, which was was negative. This term is of order $O(\de)$ that will not affect to the proof of the H\"older continuity.  The properties $(i)$ and $(ii)$ tell us that the convolution-like $\Psi_\de u$ produces a sequence of quasi-psh functions, or more precisely, they are $[1+ O(\de)]\om$-psh. Furthermore, they are comparable in the uniform norm with the usual convolution.
By \cite[Remark~4.6]{De94} 
\[\label{eq:convoution-compare}
\Psi_\de u = u*\chi_\de + O(\de^2)
\] for $0<\de<\de_0$ for a small $\de_0>0$ in the normal coordinate system center at $z$, where $\{\chi_\de\}_{\de>0}$ is the standard radial smooth kernel family. In this coordinate system $(U, z)$ with $z: U\to\Om \subset \bC^n$ a biholomorphic, we also consider
the average value over an Euclidean $\de$-ball of $u$
\[\label{eq:av-loc}	\cw u_\de (z) = \frac{1}{v_{2n} \de^{n}} \int_{B_\de(z)} u (x) dV_{2n}(x),
\]
where  $B_\de (z) = \{x \in \bC^n: |z-x| < \de\}$. Then, it follows from  \cite[Eq. (5.11), (5.12)]{KN23}  that there exists a uniform constant $C>0$ such that in $\Om_\de = \{z\in \Om: {\rm dist}(z,\d\Om) >\de\}$
\begin{align} \label{eq:cov-reg1}
	u *\chi_\de - u &\leq  C (\cw u_\de  -u), \\
	\cw u_{\de/2}  -u &\leq C (u *\chi_\de - u) \label{eq:cov-reg2}.
\end{align}
Combining estimates \eqref{eq:convoution-compare}, \eqref{eq:cov-reg1} and \eqref{eq:cov-reg2} and the line of argument given in \cite[Lemma~4.2]{GKZ08} (see also \cite[Theorem~3.2]{Z21}) we arrive at

\begin{cor}\label{cor:test-h} Let $u\in PSH(X,\om) \cap L^\infty(X)$. Assume there exist uniform constants $C, \de_0>0$ and $0<\mu \leq 1$ such that 
$$
	\Psi_\de u - u \leq C\de^\mu \quad \text{on } X, \quad 0<\de \leq \de_0.
$$Then $u$ is $\mu$-H\"older continuous on $X$.
\end{cor}

\subsection{Characterizations of H\"older continuity}
First, we will show that the modulus of continuity of $V_K$ on $X$ is equivalent to the one on $K$ only. This is a generalization of a well-known result of B\l ocki \cite[Proposition~3.5]{Siciak97}.  For simplicity we only prove it for the H\"older continuity case.

Let $K\subset X$  be a compact subset and $a \in K$.  For $0<\de \leq 1$ the modulus of continuity of $V_K$ at $a$ is given by
\[\notag
	\vpi_K(a,\de):= \sup_{|z-a| \leq \de} V_K(z).
\]
Then the modulus of continuity of $V_K$ on $K$ is given by
\[\label{eq:mod-cpt}
	\vpi_K(\de) = \sup \{\vpi_K(a,\de): a\in K\}.
\]

\begin{lem}\label{lem:property} Let $K\subset X$ be a non-pluripolar compact subset and $0<\mu \leq 1$. Then,
 $V_K$ is $\mu$-H\"older continuous on $X$  if and only if  it is $\mu$-H\"older continuous on $K$, i.e., there exists $0<\de_0 \leq 1$ such that for every $0<\de \leq \de_0$,
$$	V_K (z) \leq C \de^\mu, \quad {\rm dist}(z, K) \leq \de.
$$
\end{lem}

\begin{proof} The necessary condition is obvious. To prove the sufficient condition, we use the regularization theorem of Demailly (Lemma~\ref{lem:kis}). Notice from \eqref{eq:mod-cpt} that the H\"older continuity of $V_K$ on $K$ is equivalent to 
$$	\vpi_K(\de) \leq C \de^\mu
$$
for every $0<\de \leq \de_0$, where $C$ is a uniform constant and $\de_0>0$ is a fixed small constant. 
In particular,   $V_K^* \equiv 0$ on $K$.  Therefore, $V:=V_K$ is continuous (Corollary~\ref{cor:semicontinuity}). Consider the regularization $\Psi_t V$  of $\om$-psh function $V$ as in \eqref{eq:phie}  and
$$
	V_{\de, b} (z) = \inf_{[0,\de]} \left(\Psi_t V(z) + c_1 t^2 + c_1 t - b \log \frac{t}{\de}\right).
$$
From Lemma~\ref{lem:kis} we know that  $\Psi_tV + c_1t^2$ is increasing in $t$ and
$$
	\om + dd^c V_{\de, b} \geq -(c_0 b + 2c_1 \de) \om.
$$
Here $c_0, c_1$ are uniform constants depending only on $X$ and $\om$. 
Consider $b = A\de^\mu $ for $A>0$ so that $c_0 b+ 2c_1\de = \de^\mu$. Then
\[\label{eq:modify}
	V_\de = \frac{V_{\de,b}}{1 + \de^\mu} \in PSH(X,\om).
\]
Moreover, for $a\in K$,
\[\label{eq:bound1}\begin{aligned}
	(1+ \de^\mu) V_\de(a)
&\leq  \Psi_\de V(z) + c_1 \de^2 + c_1\de \\
&\leq \sup_{|z-a|\leq \de}	V(z) + c_1 \de + c_1\de^2 \\
&\leq \vpi_K(\de)+ c_1 \de +c_1\de^2 \\
&\leq c_2 \de^\mu,
\end{aligned}\]
where in the last inequality  we  used the fact that $V$ is $\mu$-H\"older continuous on $K$ and $c_2$ is a uniform constant.
Moreover, $V\geq 0$ on $X$, so we have $V_{\de}(z) \geq 0$. Therefore,
$$
	V_\de(a)  \leq c_2\de^\mu \quad \text{for } a\in K.
$$
This combined with the $\om$-psh property in \eqref{eq:modify} and the definition of $V$ gives
\[\label{eq:bound2}
	V_\de (z) \leq V(z) + c_2\de^\mu \quad\text{for } z\in X.
\]
At this point we can conclude the H\"older continuity of $V$ on $X$  as in the argument in \cite{DDGHKZ}. Since our setting is quite different, 
we give all details for the reader's convenience.

Let us fix  a point $z\in X$, then minimum in the definition of $V_{\delta,b}(z)$  is realized for  some   $t_0 = t_0 (z)$. By \eqref{eq:modify} and \eqref{eq:bound2} we have
\[\notag
	(1+ \delta^\mu) (\Psi_{t_0} V + c_1t_0^2 + c_1t_0 - b \log \frac{t_0 }{\delta} - V) \leq c_2\delta^\mu.
\]
Since $\Psi_t V + c_1t^2 + c_1t - V \geq 0$, we have
\[\notag 
	b (1 + \delta^\mu) \log\frac{t_0}{\delta} \geq - c_2 \delta^\mu.
\]
Combining this with $b = A\de^\mu $, one gets that
$$
	t_0(z) \geq \delta \kappa \quad \mbox{ for } \kappa = \exp \left(- \frac{2Ac_2}{(1+ \delta_0^\mu)}\right),
$$
where $\delta_0$ is already fixed at the beginning, and therefore $\kappa$  is a  uniform constant.
Since  $t \mapsto \Psi_t V + c_1t^2$ is increasing and $t_0 :=t_0 (z)\geq \delta \kappa$, 
\[\begin{aligned} \notag
	\Psi_{\kappa \delta} V(z)  + c_1 (\delta \kappa)^2 + c_1(\delta \kappa) - V(z) 
&	\leq \Psi_{t_0} V (z) + c_1t_0^2 + c_1 t_0 -V (z)  \\
&	= V_{\delta,b} (z) - V (z) \\
&	= \frac{\delta^\mu}{1-\delta^\mu} V_\delta  + (V_\delta -V), 
\end{aligned}\]
where the last equality follows from \eqref{eq:modify}. Combining this and \eqref{eq:bound2}  we get that
$
	\Psi_{\kappa \delta} V (z) - V(z) \leq C \delta^{\mu}. 
$
Rescaling $\delta := \kappa \delta$ and increasing the uniform constant $C$ implies
\[\label{eq:bound3}
\Psi_\de V(z) - V(z) \leq C \de^{\mu}, \quad 0<\de \leq \de_0
\]
(shrinking $\de_0$ if necessary). By Corollary~\ref{cor:test-h} the H\"older continuity of $V$ follows.
\end{proof}

\begin{remark} For the projective space $\bP^n$ equipped  with   the Fubini-Study metric,  Lemma~\ref{lem:property} answered a question of Sadullaev \cite[Problem~2.13]{Sa16}.
\end{remark}

We have now  the weighted version.
\begin{prop} \label{prop:W}
Let $K\subset X$ be a non-pluripolar compact set and let $\phi$
 be  H\"older continuous. Then, $V_{K,\phi}$ is H\"older continuous on $X$ if and only if it is H\"older continuous on $K$, i.e., for every $z\in X$ with ${\rm dist}(z, K) \leq \de \leq \de_0$, 
$$ V_{K,\phi} (z) \leq \phi(z) + C \de^\mu,
$$
where $C$ and $0<\de_0\leq 1$ are  uniform constants.
\end{prop}

\begin{proof} Let us write $V:= V_{K,\phi}$. The necessary condition easily follows from the H\"older continuity of $V$ and the characterization Lemma~\ref{lem:weight-vs-unweight}-(c). The proof of the sufficient condition  is very similar to the one in Lemma~\ref{lem:property}, namely, we consider the function regularization $\Psi_t V$ of  $V$ and keep the notations as in the proof of that lemma. The equation \eqref{eq:bound1} now becomes, for $a\in K$,
$$\begin{aligned}
	(1+ \de^\mu) V_\de(a)
&\leq  \Psi_\de V(a) + c_1 \de^2 + c_1\de \\
&\leq \sup_{|z-a|\leq \de}	V(z) + 2c_1 \de \\
&\leq \phi(a) + C \de^\mu,
\end{aligned}$$
where in the last inequality  we  used the fact $V(z) \leq \phi(z) + C\de^\mu$ for ${\rm dist}(z,K) \leq \de$ and $\phi$ is $\mu$-H\"older continuous (we may decrease $\mu>0$ if necessary). Therefore, $V_\de(a) \leq \phi(a) + (\|V_\de\|_{L^\infty} + C) \de^\mu$. Now by the definition of $V$,
$$
	V_\de(z) \leq V (z) + c_2\de^\mu,
$$
where $c_2$ depends additionally on the sup-norm of $V$. This is precisely \eqref{eq:bound2}.
Therefore, $V$ is H\"older continuous on $X$ by the proof of the lemma.
\end{proof}

Let us prove the first characterization in Theorem~\ref{thm:characterization-hcp}.

\begin{proof}[Proof of Theorem~\ref{thm:characterization-hcp}-(a)]
Assume $V_K(z) \leq C \de^\mu$ for every $z \in X$ that ${\rm dist}(z,a) \leq \de \leq \de_0$, where $0<\de_0 \leq 1$ is fixed. This implies $V_{K}^*(z) \leq C \de^\mu$ for every $z$ such that ${\rm dist}(z,a) \leq \de/2$. In particular, $V_K$ is continuous. 

Next, as in the proof of Theorem~\ref{thm:characterization-c}-(a) we write $B = B(a,r)$ and 
let $0\leq \chi \leq 1$ be a cut-off function such that $\chi(a)=0$ and $\chi \equiv 1$ on $X\setminus B$. Notice its $C^1$ and $C^2$-norms are controlled by $c_1/r$ and $c_2/r^2$, respectively. Furthermore, for $\veps = c_0 r^2 \leq 1/2$ with a small $c_0>0$ depending only on $\om$ we have $-\veps\chi \in PSH(X,\om/2)$. 
Notice that  $\chi (z) \leq c_1 \de /r$ for ${\rm dist}(z,a) \leq \de$. Hence, using \eqref{eq:basic-ineq} we get
$$
	\veps  h_{K\cap  B}^* (z) \leq \veps \chi(z) + C \de^\mu \leq (\frac{c_1\veps}{r} + C) \de^\mu,
$$
where $C, c_1$ do not depend on $r$. Combining this with \eqref{eq:zero-one-function} we get 
  for ${\rm dist}(z, a) \leq \de/2$,
\[\label{eq:H-norm}
	V_{K\cap B}^* (z) \leq M \left(\frac{c_1}{ r} + \frac{C}{\veps} \right)  \de^\mu,
\]
where $M = \sup_XV_{K\cap B}^*$.
The proof is completed.
\end{proof}

The above proof indeed gives the following precise estimate 
\begin{cor} \label{cor:sharp-Holder-norm} Let $K, a, B(a,r)$ be as in Theorem~\ref{thm:characterization-hcp} and $V_K$ is continuous. 
\begin{itemize}
\item
[(a)] Assume $V_K$ is $\mu$-H\"older continuous at $a$. Then, for $0<\de \leq \de_0$,
\[\label{eq:norm-loc}
	V_{K \cap  B(a,r)} (z) \leq \frac{A}{cap(K \cap B(a,r))} \frac{ \de^\mu }{r^2}, \quad  {\rm dist}(z, a) \leq \de,
\]
where $A$ is a uniform constant that  is independent of $a$ and $r$. 
\item
[(b)] Conversely, if there exist uniform constants $0<\mu, r_0, \de_0 \leq 1$  and $A>0$ such that \eqref{eq:norm-loc} holds for every $a\in K$, $0<r \leq r_0$ and ${\rm dist}(z, a) \leq \de \leq \de_0$, then  $V_K$ is $\mu$-H\"older continuous.
\end{itemize}
\end{cor}

\begin{proof} The inequality in (a) is a direct consequence of \eqref{eq:H-norm} as $\veps = c_0r^2$ and the inequality $\sup_{X} V_{K\cap B}^* \leq A/cap(K\cap B)$ in \eqref{eq:cap-comparison}. 
The statement in (b) follows from a covering argument as follows. Let $z \in X$ be a point that ${\rm dist}(z, K) \leq \de$. Let $a\in K$ be such that ${\rm dist}(z, K) = {\rm dist}(z,a)$. Given $0<r\leq r_0$, we can cover $K$ by finitely many $B(a_i,r/2)$, $i\in I$. Then $a\in B(a_i,r/2)$ for some $a_i$. Hence, $B(a_i, r/2) \subset B(a,r)$. Put $c = \min_{i\in I} cap (K\cap B(a_i,r/2))>0$. We have 
$$
	cap(K \cap B(a,r)) \geq cap(K \cap B(a_i,r/2)) \geq c.
$$
It follows from monotonicity and the assumption \eqref{eq:norm-loc}  that 
$$
	V_K(z) \leq V_{K\cap B(a,r)} (z) \leq \frac{A}{c} \frac{\de^\mu}{r^2}.
$$
The proof is completed as $r$ is fixed.
\end{proof}

\begin{proof}[Proof of Theorem~\ref{thm:characterization-hcp}-(b)]
Assume that $V_K$ and $\phi$ are H\"older continuous. We wish to show that $V:= V_{K,\phi}$ is H\"older continuous. Without loss of generality, we may assume that $\phi$ is H\"older continuous with the same exponent $0<\mu \leq 1$. Otherwise, we just take the minimum of  two exponents. It follows from Theorem~\ref{thm:characterization-c}-(b) and Lemma~\ref{lem:weight-vs-unweight}-(c) that  $V\leq \phi$ on $K$ and it is continuous. To get the H\"older continuity of $V$, by Proposition~\ref{prop:W}, it is sufficient to prove 
$$
	V(z) \leq \phi(z) + C\de^\mu, \quad {\rm dist}(z, K) \leq \de \leq \de_0,
$$
for uniform constants $C$ and $0< \de_0\leq 1$. Indeed, 
fix such a point $w$ and let $a\in K$ be such that $0<{\rm dist}(w, a) = {\rm dist}(w, K) \leq \de$. Let $r>0$ be small and its value will be determined later. By H\"older continuity, $|\phi(z)- \phi(a)| \leq c_3 r^\mu$ on $B(a,r)$. Using this and the monotonicity \eqref{eq:monotonicity}, we obtain
\[\label{eq:basic-inequality-holder1}\begin{aligned}
	V  &\leq V_{K \cap  B(a,r), \phi(a) + c_3r^\mu} \\&= \phi(a) + c_3r^\mu + V_{K\cap B(a,r)} \\&\leq \phi(w) +  c_3 \de^\mu+ c_3r^\mu + V_{K\cap B(a,r)}.
\end{aligned}\]
By Corollary~\ref{cor:sharp-Holder-norm}-(a) and uniform density in capacity of $K$ we derive 
\[\label{eq:basic-inequality-holder2}
	V_{K\cap B(a,r)}(w) \leq \frac{A}{cap(K \cap B(a,r))} \frac{\de^\mu}{r^2} \leq \frac{A \, \de^\mu}{ \vka \,r^{q+2}}
\]
for uniform constants $A, \vka,q$ which are independent of  $r$ and $a$.
Now, we can choose $r = \de^\frac{\mu}{\mu+2+q}$ to conclude $V(w) \leq \phi(w) + C \de^{\mu'} $, where $\mu' = \frac{\mu^2}{\mu+2+q}$. Hence, $V$ is H\"older continuous on $X$.
\end{proof}

\begin{remark} \label{rmk:exponent-W}The above proof showed that for a compact $K$ satisfying the uniform density in capacity in \eqref{eq:cap-density} if $V_K$ is $\mu$-H\"older continuous, then $V_{K,\phi}$ is $\mu'$-H\"older continuous for $\mu' = \mu^2/(\mu+2+q)$.
\end{remark}

\begin{remark} Similar to Corollary~\ref{cor:loc-w-c} the converse direction of Theorem~\ref{thm:characterization-hcp}-(b) holds for the H\"older continuous weight $\phi$ which is a quasi-plurisubharmonic function. In particular, it holds if  $\phi$ is  $C^2$-smooth. Since the proof is very similar to the one of that corollary, we omit it. 
\end{remark}

\begin{proof}[Proof of Theorem~\ref{thm:characterization-hcp}-(c)]  
 For the sufficient condition, let us assume $K$ has local $\mu$-HCP of order $q$. This means that for every $a\in K$ and $0<\de \leq 1$,
\[\label{eq:basic-ineq-holder-loc1}
	L_{K\cap B(a,r)} (z) \leq \frac{C\de^\mu}{r^q}, \quad {\rm dist}(z, a) \leq \de.
\]
Then, Lemma~\ref{lem:cap-density} implies that  $K$ has a uniform density in capacity. Furthermore, by the monotonicity, H\"older continuity of $\rho$ and \eqref{eq:basic-ineq-holder-loc1} we infer
$$
	L_{K,\rho} (z) \leq L_{K\cap B(a,r), \rho(a) + c_3 r} = \rho(a) + c_3 r + L_{K\cap B(a,r)} \leq \rho(a) + c_3 r + \frac{C\de^\mu}{r^q},
$$
where $c_3$ is the Lipschitz norm of $\rho$ on the ball $B(0,R)$ containing $K$.
Hence, for ${\rm dist}(z, K) \leq \de$ and $0< r <r_0$,
$$
	L_{K,\rho}(z) \leq \rho(z) + 2 c_3 r + \frac{C\de^\mu}{r^q}.
$$
Now,  we can choose $r = \de^{\eps}$ with $\eps = \mu/(q+1)$ to conclude that 
$$L_{K,\rho}(z) - \rho(z) \leq C \de^{\mu'}, \quad \mu' = \frac{\mu^2}{1+q}.$$
Hence, the H\"older continuity of $V_K$ follows from the identity \eqref{eq:identity-LV}:
\[\label{eq:identity}
	L_{K,\rho} (z) =V_{K} (z) + \rho (z), \quad z\in \bC^n.
\]

Conversely,  assume $V_K$ is $\mu$-H\"older continuous on $\bP^n$ and $K$ has a uniform density in capacity. Let $a\in K$ and fix a ball $B:= B(a,r_0)$ as in the proof of Theorem~\ref{thm:characterization-c}-(c). By the comparisons   between the extremal functions \eqref{eq:weight-vs-unweight} and \eqref{eq:identity}, we have for $0<r<r_0$,
$$
	L_{K\cap  B}(z) \leq L_{K\cap  B_r,\rho}(z) - \inf_{K\cap  B_r} \rho  =  V_{K\cap B_r} (z) + \rho(z) - \inf_{K\cap  B_r} \rho.
$$
The right hand side can be estimated by Corollary~\ref{cor:sharp-Holder-norm}-(a) and the uniform density in capacity as follows. We have for ${\rm dist}(z, a) \leq \de \leq \de_0$, 
$$V_{K\cap B_r}(z) \leq \frac{A}{cap(K \cap B_r)}\frac{\de^\mu}{ r^2} \leq \frac{A}{\vka} \frac{\de^\mu}{r^{2+q}}.$$ 
Observe also that $$\begin{aligned} \rho (z) - \inf_{K\cap B_r}\rho &= \rho(z) -\rho(x) \\& \leq c_3 (|x-a| + |z-a|) \\ 
& \leq c_3 (r + \de),
 \end{aligned}
 $$
 where $\rho(x) = \min_{B(a,r)} \rho$. Altogether we obtain for every $0<r\leq r_0$,
 $$
 	L_{K\cap B} \leq \frac{A}{\vka} \frac{\de^\mu}{r^{2+q}} + c_3(r+\de).
 $$
Choosing $r = r_0\de^\frac{\mu}{3+q}$ we conclude that $L_{K\cap B}(z) \leq C \de^\frac{\mu}{3+q}/r_0^2$. Notice that the constant $C= C(A, c_3)$ is independent of the point $a$ and $r$. The proof of the necessary condition in (b) follows.
\end{proof}

\section{Regularity of compact sets}
\label{sec:regularity}

Thanks to the local characterizations in Theorem~\ref{thm:characterization-c} and Theorem~\ref{thm:characterization-hcp} for studying the regularity of the extremal functions we may assume that the compact set is contained in a holomorphic coordinate unit ball. Without loss of generality, we may restrict ourself to compact subsets in $\bC^n \subset \bP^n =:X$. We provide a large number of examples of compact sets possessing these properties.

\subsection{Locally $L$-regular sets}
\label{sec:local-regularity}
In view of Theorem~\ref{thm:characterization-c}-(c) and the applications in \cite{BBW11} and \cite{MV22} we make an effort to collect in this section many well-known examples of locally $L$-regular sets. This is a classical topic but the results are scattered in many different places.

\begin{expl}[accessibility criterion]\label{expl:acc} This criterion is due to  Ple\'sinak \cite{Pl79}. It provides the following typical example.
Let  $\Om \subset \bC^n$ be an open bounded subset with $C^1$-boundary. Then, $\bar\Om$ is locally $L$-regular (see also \cite[Corollary~5.3.13]{Kl91} and \cite{Dieu}).  This criterion has been generalized in \cite{Pl84} for subanalytic subset in $\bC^n$ and later in \cite{Pl01} for more general setting. 
\end{expl}

\begin{expl}[Siciak] 
\label{expl:siciak97a}
Let us denote $\bK$ for either $\bK=\bR$ or $\bK=\bC$. Let $B(x,r)$ be a closed ball in $\bK^n$ with center at $x$ and radius $r>0$. Let $E\subset \bK^n$ be a compact subset. 
\begin{itemize}
\item
[(a)] {\em Cusps:} let $h: [0,1]\to \bK^n$ and $r:[0,1] \to \bR^+$ be a continuous function such that $r(0)=0$ and $r(t)>0$, $0<t\leq 1$. Let $a = h(0)$.
A  cusp with vertex $a$ is a compact subset of $\bK^n$ given by 
$$
	C(h, r):= \bigcup_{0\leq t\leq 1} B (h(t), r(t)).
$$
We say that  $E$ has a cusp $C(h, r)$ at $a=h(0) \in E$ if $C(h,r) \subset E$. 
It is proved \cite[Proposition~7.6]{Siciak97} that if $E$ has a cusp $C(h,r)$ at $a\in E$  such that $r(t) = Mt^m$ and $|h(t)-h(0)|\leq At^q$, $0\leq t\leq 1$, where $M,m, A$ and $q$ are positive uniform constants, then $E$ is locally $L$-regular at $a$.

\item[(b)] {\em Corkscrew:}
 $E$ is said to have a corkscrew of order $s>0$ at $a\in E$ if there exists $r_0 \in (0,1)$ such that for every $0<r<r_0$ we can find $a'\in \bK$ for which $B (a', r^s) \subset B(a,r) \cap E$. Let $C(h,r)$ be a  cusp at $a=h(0)$. By the triangle inequality
$$
	|x - h(0)| \leq |x-h(t)| + |h(t) - h(0)|,
$$
it follows that if $r(t) = Mt^m$ and $|h(t)-h(0)| \leq At^q$, then this cusp has a corkscrew of order $s = m/\min\{m, q\} \geq 1$. 
The conclusion is that if $E$ has a corkscrew of order $s>0$ at $a\in E$, then it is locally $L$-regular at $a$ \cite[Proposition~7.10]{Siciak97}.
\end{itemize}
\end{expl}

\begin{remark} There is still missing a good characterization of locally $L$-regular compact sets in terms of capacity. The criterion in Theorem~\ref{thm:characterization-c}-(c) may provide an approach from global pluripotential theory for this problem.
It is worth to recall from Cegrell \cite{Ce82} that 
if $E \subset \bR^n \equiv \bR^n + i \cdot 0 \subset \bC^n$ is a compact subset, then the $L$-regularity and the local $L$-regularity coincide.
\end{remark}

\subsection{Local {\rm $\mu$-HCP} sets}
\label{sec:HCP}

Here we require on the local HCP with a precise estimate of the H\"older coefficient. 
Let $E\subset \bK^n$ be a compact subset ($\bK = \bR$ or $\bK = \bC$) and $a\in E$. Following Siciak \cite{Siciak85} we consider the following 

\begin{defn}[Condition {\bf (P)}] For each point $a = (a_1,...,a_n)\in E$ there exist compact connected  subsets $\ell_1,...,\ell_n \subset \bC$ and an affine non-singular mapping $h: \bC^n \to \bC^n$ satisfying
\begin{itemize}
\item
[(i)] $a \in h(\ell_1 \times \cdots \times \ell_n) \subset E$;
\item
[(ii)] $\| \ell_j \| \geq d>0$, $j =1,...,n$ ;
\item
[(iii)] $\|Dh\| \geq m >0$,
\end{itemize}
where $\| \ell_j \|$ is the diameter of $\ell_j$, the constants $d$ and $m$ do not depend on $a$. 
\end{defn}
Note that  $Dh$ denotes the Fr\'echet derivative of $h$ which is a linear map of $\bC^n$ to $\bC^n$ and does not depend on the point $a$. Hence, its norm $\|Dh\|$ does not depend on the point.  
Roughly speaking  the condition ({\bf P}) for $E$ means at each point $a\in E$ there is an {\em affine cube} of uniform size with vertex at $a$ and it is contained in $E$.

A basic result \cite[Proposition~5.1]{Siciak85} says that for compact subset $E\subset \bC^n$ satisfying the condition (P),
\[\label{eq:siciak-HCP}
	L_E(z) \leq \frac{4 \sqrt{1+\|E\|}}{md} \de^\frac{1}{2}
\] 
holds for every  ${\rm dist}(z, E) \leq \de$ and  $0<\de\leq 1$, where $\|E\|$ is the diameter of $E$.

This precise estimate of both the H\"older coefficient and the  H\"older exponent  allows us to study the local $\mu$-HCP of many classes of compact sets in $\bK^n$. Let us describe them again emphasizing on the locality.

\begin{expl} \label{expl:local-HCP} Let $\Om \subset \bK^n$ be a bounded domain. Let $E\subset \bK^n$ be a compact set.
\begin{itemize}
\item[(a)] {\em Lipschitz domain}:  Assume $\Om$ has  Lipschitz boundary and $a'\in \ov\Om$. Then for every $r>0$, the compact set $\ov \Om \cap B(a',r)$ satisfies the condition ({\bf P}) at each point.

\item[(b)] {\em Geometrical condition}: Assume there exists $r>0$ such that for every $a\in E$, there is a point $a'\in E$ for which the convex hull of the set $\{a\} \cup B(a', r)$ is contained in $E$. Then, $E$ has local $\frac{1}{2}$-HCP of order $q=n$.

\item[(c)] {\em Uniform interior sphere condition}:  there exists $r>0$ such that for each $a\in \d\Om$ there is $B(a',r)\subset \Om$ and $B(a',r) \cap (\bK^n \setminus \Om) = \{a\}$. Then, $\ov\Om$ has local  $\mu$-HCP with the exponent $\mu=1$ and of order $q=1$ if $\bK = \bC$; or with exponent $\mu=1/2$ and of order $q=2$ if $\bK=\bR$.
\end{itemize}
\end{expl}

Let us give the explanation of the above examples.

(a) For the bounded domains with Lipschitz boundary it satisfies the so called Property (${\bf H_2}$) (see \cite[page 166-167]{BG71}) which is a local property at a point. Namely there exists a non-empty parallelepiped $\pi_0$ such that each point $a\in \d\Om$ is the vertex of a parallelepiped $\pi_a$ congruent with $\pi_0$ (with respect to orthogonal transformation and translation) satisfying $\pi_a\subset \ov\Om$. This property clearly implies the condition (P) at $a\in \d\Om$ for $\ov\Om\cap B(a',r)$. Together with the estimate \eqref{eq:siciak-HCP} we conclude the local $\frac{1}{2}$-HCP of order $q=1$ of such subsets. 

(b) Next, by \cite[Propostion~5.3]{Siciak85} the geometric condition implies the condition ({\bf P}) with $\ell_i =[0,1] \subset \bC$, $i=1,...,n$, and
\[\label{eq:GC}
	m = \left(\frac{r}{n^{3}}\right)^n \frac{1}{\|E\|^{n-1}}.
\]

(c) Finally, the uniform interior sphere condition implies the geometric condition. However, we will give a proof of the improvement of the exponent and order in Lemma~\ref{lem:uniform-sphere} below.

One important case of the geometrical condition is that 
it is satisfied for any convex compact set with non-void interior in $\bK^n$, where $\bK = \bR$ or $\bK =\bC$. Hence, by
 \eqref{eq:siciak-HCP} and \eqref{eq:GC} (see also \cite[Remark~5.4]{Siciak85}) we have 
\begin{cor}[Siciak]
\label{cor:convex}
A convex compact subset in $\bK^n$ satisfying the geometrical condition has  local {\rm $\mu$-HCP} with the (optimal) exponent $\mu=1/2$ and of order $q= n$. 
\end{cor}

Let us give the proof of the statement in Example~\ref{expl:local-HCP}-(c).
Notice that the uniform interior sphere condition is satisfied for all smooth bounded domains with $C^{1,1}$-boundary. 

\begin{lem}\label{lem:uniform-sphere} Let $\Om\subset \bK^n$ be a bounded domain satisfying the uniform interior sphere condition. 
\begin{itemize}
\item
[(a)] If $\bK = \bC$, then $\ov\Om$ has local {\rm $\mu$-HCP} with the optimal H\"older exponent $\mu=1$ and of order 1.
\item
[(b)] If $\bK= \bR$ and $\bR^n \equiv \bR^n + i \cdot 0 \subset \bC^n$, then $\ov\Om$ has local {\rm $\mu$-HCP}  with the optimal exponent $\mu =1/2$ and of order $q=2$.
\end{itemize}
\end{lem}

\begin{proof} (a) Let $r>0$ be fixed.  The uniform interior sphere condition means that there exists a closed ball $B(a', r_0)$ such that 
$$
	 B(a',r_0) \cap (\bC^{n} \setminus \Om) = \{a\},
$$
where $r_0>0$ is a uniform constant. In particular, $|a' - a|  =  r_0$. By decreasing $r_0$ we may assume $r/2 \leq r_0 \leq r$ and by dilating this ball we may assume that $B(a',r_0) \subset \Om \cap  B(a,r)$. Hence,
$$
	L_{\ov\Om \cap  B(a,r)} (z) \leq L_{B(a',r_0)}(z) = \max\{\log (|z-a'|/r_0),0\},
$$
where the second identity used the explicit formula of the extremal function (see, e.g. \cite[Example~5.1.1]{Kl91}). 
For $w \in \bC^n$ with  ${\rm dist} (w,a) \leq  \de$ with small $\de$,  we have $$|w-a'| \leq |w-a| + |a-a'| \leq \de + r_0.$$ It follows that 
$$
	L_{\ov\Om \cap B(a,r)}(w) \leq \log (1 + \de/r_0) \leq \de/r_0 \leq 2\de/r.
$$
This means that $L_{\ov\Om\cap B(a,r)}$ is  Lipschitz continuous at $a$. Furthermore, the Lipschitz norm at that point  is independent of the point. Thus, $\ov\Om$ has local $\mu$-HCP with the exponent $\mu=1$ and of order $q=1$.

(b) Assume now $\ov\Om \subset \bR^n + i \cdot 0 \subset \bC^n$.  The proof goes  along the same lines as above.  Noticing that if $a\in \ov\Om$, then $B(a,r) \cap (\bR^n + i \cdot 0)$ is  the real ball $\wt B(a,r) \subset \bR^n$ and we  have an explicit  formula (see e.g. \cite[Theorem~5.4.6]{Kl91}) for such a ball via $L_{\wt B(a,r)} (z) = L_{\wt B(0,1)} (f(z))$, where
$$
	L_{\wt B(0,1)} (z)= \frac{1}{2} \log \left({\bf h} \left( |z|^2 +|\lc z, \bar z\rc  -1| \right) \right), \quad f(z) = (z-a)/r,
$$ 
and ${\bf h}(x) = x + (x^2 -1)^\frac{1}{2}$ for $x\geq 1$.
\end{proof}

Notice that there are other kind of examples arising from Example~\ref{expl:siciak97a} for cusp and corksrcew sets, because it is possible to derive precise estimate on the H\"older coefficient there. We refer the reader to \cite[Proposition~6.5]{Siciak97} for more detail.

\subsection{Uniformly polynomial cuspidal sets} 
\label{sec:UPC}
In this section we follow the method of Paw\l ucki and Ple\'sniak \cite[Theorem~4.1]{PP86} to study the uniformly polynomial cuspidal sets. The improvement is  a precise estimate on the H\"older coefficient. These sets contain, for example,  all bounded convex sets with non-void interior in $\bK^n$ (here $\bK = \bR$ or $\bK = \bC$)  and all bounded domains in $\bK^n$ with Lipschitz boundary (see Example~\ref{expl:local-HCP} and also \cite[page 469]{PP86}). 

Now we are focusing on the compact sets with cusps in $\bR^n$ considered as a natural subset in  $\bC^n$.
A compact subset $E \subset \bR^n$  is called {\em uniformly polynomial cuspidal} (UPC for short) if there exist positive constants $M$, $m$ and a positive integer $d$ such that for each $x\in E$, one may choose a polynomial map 
$$
	h_x: \bR \to \bR^n, \quad \deg h_x \leq d
$$
satisfying
$$\begin{aligned}
&	h_x(0) = x \quad \text{and}\quad h_x([0,1]) \subset E; \\
&	{\rm dist } (h_x(t), \bR^n \setminus E) \geq M t^m  \quad \text{for all } x\in E,\text{ and } t\in [0,1],
\end{aligned}$$

An important property of the UPC sets is that if $a \in E$, then 
\[ \label{eq:cusp}
	 E_a := \bigcup_{0\leq t \leq 1}  D(h_a(t), Mt^m) \subset E,
\]
where 
$D(p, r) = \{ x\in \bR^n: |x_1-p_1| \leq r, ..., |x_n-p_n| \leq r\}$ denotes the closed cube.
Notice that $E_a$ is also a closed subset.
\begin{remark} \mbox{}
\begin{itemize}
\item
[(a)] Without loss of generality we may assume that the exponent $m$ is a positive integer. Otherwise we will choose the smallest positive integer larger than $m$ instead.
\item
[(b)] It is not clear from the definition that the norm of coefficients  
$\sum_{\ell=0}^d ||h_x^{(\ell)} (0) \|$ of $h_x$ is uniformly bounded on $E$. Fortunately, in many interesting examples of UPC sets this assumption is satisfied.
\end{itemize}
\end{remark} 

\begin{thm}\label{thm:UPC} Let $E \subset \bR^n$ be a compact {\rm UPC} subset such that $\sum_{\ell=0}^{d} \|h_x^{(\ell)}(0)\|$ is uniformly bounded on $E$. Then, $E$  has local {\rm $\mu$-HCP} of some order $q$.
\end{thm}

It is worth to emphasize that by  \cite[Corollary~6.6, Remark~6.5]{PP86}, we have
all compact fat subanalytic subsets in $\bR^n$ satisfying the additional assumption. Thus, we obtain

\begin{cor}\label{cor:fat-analytic}
 A compact fat subanalytic subset in $\bR^n$ has local {\rm $\mu$-HCP} of order $q\geq 0$.
\end{cor}

This corollary combined with Theorem~\ref{thm:characterization-hcp}-(c) gives the proof of Theorem~\ref{thm:UPC-intro}.

Now, to proceed with the proof of Theorem~\ref{thm:UPC} we need the following fact

\begin{lem}\label{lem:L-polynomial} Let $E \subset \bC^k$ be a compact subset and $h: \bC^k \to \bC^n$ be a complex valued polynomial mapping of degree $d$. Then, for every $w\in \bC^k$,
$$
	L_{h(E)}(h(w)) \leq d \cdot L_{E} (w).
$$
\end{lem}

\begin{proof} Since $E$ is compact, so is $h(E)$. 
Now, let $v\in \cL(\bC^n)$ be such that $v \leq 0$ on $h(E)$. 
Since $\deg h \leq d$,
$$\begin{aligned}
	\limsup_{|w|\to +\infty}	(v\circ h(w)  - d \log |w|) 
&\leq \limsup_{|w|\to +\infty}	(v\circ h(w)  -  \log |h(w)|) +c_h \\
&\leq c_v + c_h,
\end{aligned}$$
where the second inequality used the assumption of $v \in \cL(\bC^n)$. Hence, 
$v\circ h /d \in \cL(\bC^k)$ and it is negative on $E$. It follows that $v\circ h \leq d \cdot L_E$ and this finishes the proof.
\end{proof}

\begin{rmkx}
The following interesting fact is pointed out by the referee. Even if $E$ is not pluripolar in $\bC^k$, $h(E)$ could be pluripolar in $\bC^n$. However, then $h(E)$ is an algebraic subvariety of $\bC^n$ and there is still content to the statement as in \cite{Sa76, Sa82}.
\end{rmkx}

\begin{proof}[Proof of Theorem~\ref{thm:UPC}] In what follows 
the space $\bR^n$ is identified with the subset $\bR^n+i\cdot 0$ of  $\bC^n$. Let $a\in E$ be fixed and denote by $D(a,r)$ a closed polydisc. 
Observe first that the set $E_a$ defined in \eqref{eq:cusp} satisfies
$$	E_a = \{x\in \bR^n: x = h(t) + M t^m \left(x_1^m,...,x_n^m \right), t\in [0,1], |x_i| \leq 1, i=1,...,n\}.
$$
Let $S\subset \bR \times \bR^{n}$ be the pyramid 
$$
	S = \{(t, tx_1,...,tx_n) \in \bR \times \bR^{n}:  t \in [0, 1], |x_i| \leq 1, i=1,...,n\}.
$$
This is a convex set (with non-void interior in $\bR^{n+1}$) which implies that it has HCP. The crucial observation is that
our cusp $E_a$ is the image of this set under the polynomial projection
$p(t,z): \bC \times \bC^{n} \to \bC^n$ given by
$$
	p(t,z) = h(t) + M (z_1^m, ..., z_n^m).
$$
Clearly, $p(S) = E_a$ and $p(0,0) = h(0) =a$.

To show the {\em local} $\mu$-HCP of some order at $a$  we need to shrink a bit that pyramid. We claim that for each $0< r \leq 1$, we can find $0< r' \leq r$ such that a smaller pyramid
$$
	S(r') := \{ (t,tx) \in \bR \times \bR^{n}: t\in [0,r'], |x_1| \leq r', ..., |x_n| \leq r'\}
$$
satisfies
\[\label{eq:s-pyramid}
	S(r')  \subset p^{-1}(E_a \cap D(a,r)).
\]
Indeed, for $(t,tv) \in S(r') \subset S$, the point $x = h(t) + M t^m\cdot v^m \in E_a$. Moreover,
$$\begin{aligned}
|x-a| &= |h(t) + M t^m\cdot v^m - h(0)| \\
& \leq  |h(t) - h(0)| + M t^m |v|^m \\
&\leq \left(\sum_{\ell=0}^q \|h^{(\ell)}(0)\| \right) r' + n M r', 
\end{aligned}
$$
where we used the fact that $m$ is a positive integer.
Thus we can choose
\[ \label{eq:order-UPC}
	r' = \frac{r}{\sum_{\ell=0}^q \|h^{(\ell)}(0)\| + nM}.
\]
This is the sole place we need the uniform bound for the sum $\sum_{\ell=0}^n \|h^{(\ell)}(0)\|$ that does not depend on the point $a$. Otherwise, the H\"older norm of $L_{E \cap B(a,r)}$ would depend on $a$.

Since $S(r')$ contains a ball of radius $ \tau_n r'$ in $\bR^{n+1}$ with a numerical constant $\tau_n$, it follows from Corollary~\ref{cor:convex} that $S(r')$ has local $\frac{1}{2}$-HCP of order $q=n+1$, i.e.,
\[\label{eq:HCP-convex}
	L_{S(r')} (t,v) \leq \frac{C \de^\frac{1}{2}}{r'^{n+1}}
\]
for every $(t,v) \in S_\de (r') := \{ \ze \in \bC^{n+1}: {\rm dist} (\ze, S(r')) \leq \de\}$ and $C$ does not depend on $r'$ and $\de$.

Moreover, for all $0< \de \leq  r'$, we have   $$P(\de) = \{(t,z) \in \bC\times \bC^n: |t|\leq \de, |z_i| \leq \de, i=1,...,n\} \subset S_\de(r').$$ 
Then, for such a small $\de$, the following inclusions hold
\[\label{eq:inclusion-chain}
	B (a, M \de^m) \subset p (\{0\}\times D(0,\de))\subset p(P(\de)) \subset p ( S_\de (r')).
\]

Now we are ready to conclude the local $\mu$-HCP of $ E_a$. Let $z\in \bC^n$ be such that  ${\rm dist}(z, a) \leq M \de^m$. By \eqref{eq:inclusion-chain} we have $z = p(t,v) \in \bC^n$ for some $(t,v)\in S_\de (r')$. Furthermore, by \eqref{eq:s-pyramid} we have $F:= p(S(r')) \subset E_a \cap D(a,r)$. Combining these facts  with \eqref{eq:HCP-convex} we obtain
$$\begin{aligned}
L_{ E_a\cap D(a,r)}(z) &\leq L_{F} (z) \\
&= L_{F} (p(t,v)) \\
&\leq \max(d,m) \cdot L_{S(r')} (t,v)  \\
& \leq \frac{C \de^\frac{1}{2}}{r'^{n+1}},
\end{aligned}$$
where for the third inequality we used Lemma~\ref{lem:L-polynomial}, and the last constant $C$ does not depend on $r'$ and $a$.
Rescaling $\de := M\de^m \leq r'$, we obtain
$$
	L_{E_a \cap D(a,r)} (z) \leq \frac{C \de^\frac{1}{2m}}{r'^{n+1}}
$$
for every ${\rm dist}(z,a) \leq \de$, where $0<\de \leq r'$. Notice that $r$ and $r'$ are comparable by \eqref{eq:order-UPC}. Hence, $E_a$ has local $\mu$-HCP at $a$ with the exponent $\mu = 1/2m$ and of order $q=n+1$ and so does $E\supset E_a$. This finishes the proof of the theorem.
\end{proof}

\subsection{Compact sets slided transversally by analytic half-disc} 
\label{sec:totally-real}

In this section we provide another class of compact subsets (of Lebesgue measure zero but admitting geometric structure) that have local $\mu$-HCP and of some order. We will see later that they contain generic submanifolds as important examples.

Let 
$$
	\uU_+ = \{\tau \in \bC: |\tau| \leq 1,\; {\rm Im} \tau \geq 0\}
$$
denote the (closed) upper half of the closed unit disc and we denote for $0<\de \leq 1$
$$
	{\rm U}(\de) = \{\tau \in \bC: |\tau|\leq \de\}.
$$
Motivated by \cite{SZ16} and \cite{Vu18} we consider the following class of sets.

\begin{defn} \label{defn:slide-set}
Let $E \subset \bC^n$ be a compact subset and $a\in E$. Assume that there are uniform constants $C_0, M, m$ and $\de_0>0$ (do not depend on $a$) such that for every $0<\de < \de_0$ and $x\in \bC^n$ with ${\rm dist} (x,a) \leq M\de^m$ we can find a holomorphic map $$f_a: \overset{\circ}{\rm U}_+ \to  \bC^n$$ satisfying:
\begin{itemize}
\item
[(a)] $f_a$ is  continuous on $\uU_+$ and $\sup_{U_+} |f_a| \leq C_0$;
\item
[(b)] $f_a(0) = a$ and $f_a([-1,1]) \subset E$;
\item
[(c)]
$x\in f_a (\uU_+ \cap \uU (\de)).$
\end{itemize}
For the sets satisfying the definition we say that $E$ can  be {\em slid transversally by analytic half-disc at $a\in E$}.
Moreover, $E$ is said to be slid transversally by analytic half-disc if $E$ can be slid transversally at every point $a\in E$.
\end{defn}

The condition (b) is to say that the analytic half-disc $f: \uU_+ \to \bC^n$ is attached to $E$ along the interval $[-1,1]$. The condition (c) means the analytic half-disc $f_a: \uU_+ \to \bC^n$ meets  transversally with $E$ so that $f_a(\uU_+ \cap \uU(\de))$ contains an arc joining  $a$ and $x$ (with distance $M\de^m$). 

Note that the idea of analytic disc attached to a generic submanifold in $\bC^n$ is classical in CR-geometry (see \cite{BEP99}). However, the above definition emphasizes on the quantitive estimates.
Also, in this definition it is important to require that the constants $M, m$ and $\de_0$ are independent of the point $a$. For the applications later we will need this independency of all point in the compact set. Geometrically it says that at each point we can attach transversally to the set a closed half-disc of uniform radius, consequently the family of analytic half-disc fills a neighborhood of the given point (in the ambient space). This will be clearly seen in the examples below.

In our setting the half-disc is attached along the real axis which is slightly different from  \cite{SZ16} and \cite{Vu18}. However, we can use a simple conformal map to convert the results obtained there to our setting as we are only interested near the origin. Also our definition seems to be  natural as in the following examples show. 

\begin{expl}\mbox{}
\begin{enumerate}
\item The simplest example 
$E := [-1,1] \subset \bR \subset \bC$. Then, we can choose $C_0=2$, $M=1$, $m=2$, $\de_0=1$. In fact, for 
$
	f(z) = z^2,
$
we have $f([-1,1]) = [0,1]$. Consider the analytic function $f_a: \uU_+ \to \bC$ given by
$$f_a(z) = \begin{cases}
 a + z^2 \quad \text{for } a \in [-1,0], \\
 a - z^2 \quad\text{for } a\in [0,1].
\end{cases}	
$$
Moreover, $x \in f_a( \uU_+ \cap \uU(\de))$ for every $x\in \bC$ with $|x-a| \leq \de^2$, where $0<\de <\de_0=1$. This implies that $E$ is slid transversally by analytic half-disc. We can also  see that 
$$
	A = \sup_{\tau \in [-1,1]} \frac{|f_a(\tau)-f_a(0)|}{|\tau|} \leq 1.
$$
\item
In a similar fashion we can see that a bounded domain with $C^2$-boundary or compact cube in $\bR^n = \bR^n+ i \cdot 0 \subset \bC^n$ are the sets satisfies Definition~\ref{defn:slide-set}. Using this we get another proof of the local $\mu$-HCP with the exponent $\mu = 1/2$ and of order $q=2$ (see Lemma~\ref{lem:uniform-sphere}).
\end{enumerate}
\end{expl}

We will need the following fact about the harmonic measure.
\begin{lem} \label{lem:harmonic-measure} Denote $E = [-1,1] \subset \bR \subset \bC$.
Let $h$ be the harmonic extension of $ 1-{\bf 1}_{E}$ from $\d \uU_+$  into $\uU_+$. Then, 
$$
	h(z) = \sup \left\{ v\in SH(\overset{\circ}{\rm U}_+) \cap C( \uU_+): v \leq 1 - {\bf 1}_E \text{ on  } \d\uU_+ \right\},
$$
and $h$ is Lipschitz continuous at $0 \in E$ (or on a proper compact subinterval).
\end{lem}

\begin{proof} The harmonicity of the envelope is a classical result. Moreover, $h(z)$ is given by  an explicit formula 
$$
	h(z) = \frac{2}{\pi} \arg \left(\frac{1+z}{1-z} \right).
$$
This function  is clearly Lipschitz near $0$. 
\end{proof}

It is a classical fact that  $$L_{[-1,1]} (z) = \log \left|z + \sqrt{z^2-1}\right|,$$
where the branch of the square root is chosen so that $|z+ \sqrt{z^2-1}|\geq 1$ in the whole plane  (see. e.g. \cite[Lemma~5.4.2]{Kl91}). This function is (optimal) $\frac{1}{2}$-H\"older continuous.
We obtain the following generalization.

\begin{prop}\label{prop:1/m-Holder} Assume that $E \subset \bC^n$ is a compact set satisfying the conditions {\rm (a), (b)} and {\rm (c)} in Definition~\ref{defn:slide-set} at every point $a\in E$. Then, $L_E$ is $1/m$-H\"older continuous.
\end{prop}

\begin{proof} It is enough to prove that $L_E$ is $\frac{1}{m}$-H\"older continuous on $E$ by B\l ocki's result \eqref{eq:local-modulus-of-continuity}, i.e.,   there exists $C> 0$ such that  for every $a\in E$,
$$
	L_E(z) \leq  C |z - a|^\frac{1}{m}
$$
for  $|z-a| \leq \de.$
In fact, for $z \in \bC^n$ such that ${\rm dist}(z,a) \leq M\de^m$ there exists $\tau \in  \uU_+ \cap  \uU( \de)$ such that $z = f_a(\tau)$. Therefore,
$$
	L_{E} (z)  = L_{E} (f_a (\tau)) 
$$
Observe that $v :=L_{E} \circ f_a$ is subharmonic in the interior of $\uU_+$ and $v\in C^0(\uU_+)$. It satisfies  $\sup_{U_+} v \leq C$ for a constant depending only on $E$, and  $v \equiv 0$ on $[-1,1]$. In particular, $v/C$ is a candidate for the left hand side of the envelope in  Lemma~\ref{lem:harmonic-measure}. Hence, 
$$ v(\tau) \leq C  h(\tau)\leq  C \de,
$$
where $h(\tau)$ is the harmonic measure with respect to the interval $[-1, 1] \subset \d \De_+$ and we used the fact $|h(\tau)| \leq c_1 |\tau| \leq c_1\de$.
Combining with the above identity and  rescaling we conclude that 
$
	L_{E} (z) \leq C \de^\frac{1}{m}.
$
\end{proof}

\begin{remark} \mbox{}
\begin{itemize}
\item
[(a)] 
Suppose that the $\mu$-H\"older coefficient/norm, $0<\mu\leq 1$, of the analytic half-disc $f_a$ at $0$ is bounded by a constant $A$ that is independent of the point $a$, then  from the above proof we can easily get  the local $\mu$-H\"older regularity at $a\in E$ with the exponent $\mu = 1/m$ and of order $q = 1/\mu$, i.e., 
$$
	L_{E\cap B(a,r)} (z) \leq \frac{C A^\frac{1}{\mu}\de^\frac{1}{m}}{r^\frac{1}{\mu}}  , \quad {\rm dist}( z, a) \leq \de,
$$
where $C$ is independent of $r$ and the point $a$.

\item
[(b)]
Suppose $E$ satisfies (a), (b) and (c) in Definition~\ref{defn:slide-set} at $a\in E$ with an analytic half-disc $f_a$ whose the $\mu$-H\"older coefficient at $0$, where $0<\mu \leq 1$, 
$$A= A (\mu):=\sup_{\tau \in [-1,1]} \frac{|f_a(\tau) -f_a(0)|}{|\tau|^\mu} <C$$
for a uniform constant $C$ which does not depend on the point $a$.
Then, for every small $r>0$, the set $E \cap B(a,r)$ can be slid transversally at $a$ by an analytic half-disc $g_a: \uU_+ \to \bC^n$ given by $$g_a (\tau) = f_a\left((r/A)^\frac{1}{\mu} \tau \right).$$
Moreover, if $C_0, M, m, \de_0$ are uniform constants satisfying $(c)$ for $E$ at $a$, then the constants
$C_0,  M (r/A)^\frac{m}{\mu}, m, \de_0$ satisfy (c) for $E\cap B(a,r)$ at $a$.

\end{itemize}
\end{remark}

It is proved by Vu \cite[Proposition~2.5]{Vu18} and by Sadullaev and Zeriahi \cite{SZ16} that a $C^2$-smooth generic submanifold in $\bC^n$ (or in a complex manifold) can be slid  transversally by analytic half-disc together with uniform control of Lipschitz norm at $0$ of $f_a$.
As a consequence the extremal functions of these generic submanifolds are Lipschitz continuous. Furthermore, these manifolds have  local $\mu$-HCP with the exponent $\mu=1$ and of order $q=1$.
It will be interesting to obtain more examples of sets satisfying Definition~\ref{defn:slide-set}.

\section{Applications in equidistribution speed for Fekete points}\label{sec:app}

In dimension $n=1$, for a non-polar compact set $E\subset \bC$, the function $L_E$ coincides with the Green function of the unbounded component of $\bC \setminus E$ with the pole at infinity.  A classical result saying that  the sequence of the probability counting  measures  for Fekete points of $E$, called \em Fekete's measures\rm, converges weakly to the equilibrium measure $\mu_{\rm eq}$  of $E$ as the number of points goes to infinity (see, e.g \cite{BBL92}).  The speed of convergence of the sequence can be also quantified for a compact domain whose boundary is smooth. 

The analogous problems regarding the weak convergence and speed of convergence of the Fekete measures of a compact subset in $\bC^n$, $n\geq 1$, had been open for a long time. The weak convergence was obtained in a deep work of Boucksom, Berman and Witt-Nystrom \cite{BBW11}. The speed of convergence was proved \cite{LO16} for $K=X$ a compact projective manifold and $
\phi$ a strictly smooth plurisubharmonic weight. Later, Dinh, Ma and Nguyen \cite{DMN} obtained the estimate for a large class of compact sets that satisfies the $(\Cc^\al, \Cc^{\al'})$-regularity (Definition~\ref{defn:DMN-regular}).

Let $\cP_d(\bC^n)$ be the set of complex valued polynomials of degree at most $d$. Then its dimension is $N_d = \binom{n+d}{n}$, and let $\{e_1,...,e_{N_d}\}$ be an ordered system of all monomials $z^{\al}:= z^{\al_1}_1 \cdots z_n^{\al_n}$ with $|\al| = \al_1 +\cdots \al_n \leq d$, where $\al_i \in \bN$. For each system $x^{(d)} = \{x_1, ..., x_{N_d}\}$ of $N_d$ points of $\bC^n$ we define the generalized Vandermonde matrix ${\rm VDM} (x^{(n)})$ by
$$
	{\rm VDM} (x^{(d)}):= \det [e_i(x_j)]_{i,j=1,...,N_d}.
$$

Let $K \subset \bC^n$ be a non-pluripolar  compact subset. Following \cite{DMN} we say that  a {\em Fekete configuration of order $d$} for $K$ is  a system $\xi^{(d)}=\{\xi_{1},....,\xi_{N_d}\}$ of $N_d$ points of $K$ that maximizes the function $|{\rm VDM}(x^{(d)})|$ on $K$, i.e.,
$$	
	\left| {\rm VDM}(\xi^{(d)}) \right| = \max \left\{ \left|{\rm VDM}(x^{(d)}) \right|: x^{(d)}\subset K\right\}.
$$

Given a Fekete configuration 
$\xi^{(d)}$ of $K$, we  consider the probability measure on $\bC^n$ defined by
$$
	\mu_d := \frac{1}{N_d} \sum_{j=1}^{N_d} \de_{\xi_j},
$$
where $\de_x$ denotes the Dirac measure concentrated at the point $x$. It is called Fekete's measure of order $d$  by \cite[Definition~1.4]{DMN}. 
It is known for $n=1$ that $\{\mu_d\}$ converges weakly to the equilibrium measure $\mu_{\rm eq}$ of $K$ as $d$ goes to infinity. In a fundamental paper  Berman, Boucksom and Witt Nystrom \cite{BBW11}  proved the generalization of this result for $n\geq 2$. Namely, 
\[\label{eq:equidistribution}
	\lim_{d\to \infty}\mu_d = \mu_{\rm eq}, \qquad\text{where }  \mu_{\rm eq}= \frac{(dd^c L_K^*)^n}{\int_{\bC^n} (dd^c L_K^*)^n}, 
\]
in the weak topology of measures. This result coincides with the classical one for  $n=1$ and it was listed as an open problem in \cite[15.3]{Siciak82} and \cite[Problem~3.3]{ST97}. 

In fact this problem is considered as a special case of $K \subset \bC^n \subset \bP^n$ of a very general  reformulation in the framework of a big line bundle over a complex manifold in \cite{BBW11}. It is possible by observing  that 
the space $\cP_d(\bC^n)$ is isomorphic to the space of homogeneous polynomials of degree at most $d$ in $(n+1)$-variable $\cH_d(\bC^{n+1})$. The latter space can be identified the space of global holomorphic section $H^0(\bP^n, \cO(d))$, where $\cO(d)$ is the $d$-th tensor power of the tautological line bundle $\cO(1)$ over $\bP^n$.  We refer the readers to \cite{Le10} for a self-contained proof which derived from \cite{BB10} and \cite{BBW11} using only (weighted) pluripotential theory in $\bC^n$.  

Note again that the regularity of weighted extremal functions is a local property and it  is invariant under biholomorphic maps. Without loss of generality  we restrict ourself to the compact subset in $\bC^n$ as in Section~\ref{sec:regularity} (see also Appendix~\ref{sec:appendix}).  Hence, an immediate consequence of 
the characterizations in Theorem~\ref{thm:characterization-hcp}  is

\begin{lem}
 All local $\al$-{\rm HCP} compact subsets of order $q$ in $\bC^n\subset \bP^n$ are $(\Cc^{\al}, \Cc^{\al'})$-regular, where $\al'$ is explicitly computed in terms of $\al$ and $q$.
 \end{lem}
 
 By Remark~\ref{rmk:exponent-W} and the proof of Theorem~\ref{thm:characterization-hcp} we can compute the exponent $\al'= \al''^2/(\al''+2+q)$ where $\al'' = \al^2/(1+q)$. This combined with \cite[Theorem~1.5]{DMN} shows that we have a large number of new examples from  
 Sections~\ref{sec:HCP} \ref{sec:UPC},  \ref{sec:totally-real} for which the following estimate for speed of convergence holds. 
 
\begin{thm}\label{thm:equidistribution} Let $K \subset \bC^n$ be a local $\al$-{\rm HCP} compact subset of order $q$. Then, there exists a constant $C = C(K, \al)$ such that for every test function $v\in \Cc^{\al}$ (the H\"older space on $\bC^n$) and for every Fekete configuration $\xi^{(d)}$ of $K$, 
$$
	|\lc \mu_d - \mu_{\rm eq}, v\rc|  \leq C \|v\|_{\Cc^{\al}} d^{- \al'}.
$$
\end{thm}

Notice that the holomorphic maps which preserve  HCP are characterized in  \cite{Pi21}. It is likely that the criterion can be extended to the local $\mu$-HCP case. If it is true, then we will obtain compact sets, via nice holomorphic maps, which  are $(\Cc^\al, \Cc^{\al'})$-regular.
 
\section{Appendix}
\label{sec:appendix}

In this section we give the proof of Theorem~\ref{thm:characterization-hcp}-(c) without assuming that it is a compact set contained in $\bC^n\subset \bP^n.$

Let  $(X,\om)$ be a compact K\"ahler manifold of dimension $n$. 
Let $K\subset X$ be a non-pluripolar compact subset and $a\in K$. 
Consider a holomorphic coordinate ball $(\Om, \tau)$ in $X$ centered at $a$,  
\[\label{eq:c-ball}\tau: \Om \to \bB:=\bB(0,1) \subset \bC^n
\] is a biholomorphic map with $\tau(a) =0$. Suppose that $$\om = dd^c \rho$$ for a strictly psh function $\rho\in PSH(\Om) \cap C^\infty(\ov\Om)$, where   $\rho$ attains its minimum $\inf_{\Om}\rho = \rho(a) =0$. (In general we can do this by shrinking $\Om$ and modifying $\rho$ by a pluriharmornic function.)  
As the H\"older continuity of $V_{K}$ is independent of the metric (Remark~\ref{rmk:regularity-c}), by rescaling   we may also assume   
$$0 \leq \rho \leq 1 \quad\text{ on } \ov\Om.$$ 

We recall some classical facts in pluripotential theory. Let $E \subset \bB(0,\frac{1}{2})$ be a compact set. The weighted zero-one relative extremal function is given by
\[\label{eq:zero-one-C}\notag
	u_{E}(z) = \sup \left\{ v(z) : v\in PSH(\bB), \, v_{|_E} \leq 0, v\leq  1 \right\}.
\]
The classical result in pluripotential theory \cite[Proposition~5.3.3]{Kl91} tells us that the extremal function $L_E$ satisfies
\[\label{eq:compare-C}
	C_1 \, u_E \leq L_E \leq C_2 u_E \quad\text{in  } \bB,
\]
where $C_1, C_2$ are two positive constants depending only on $E$ and $\bB$.

Now, let $\phi$ be a continuous function on $\ov \bB$.
For simplicity we assume  that the weight function satisfies 
\[\label{eq:weight-c}
	0\leq \phi \leq 1.
\] 
The (positive) weighted relative extremal function  is defined by 
\[\label{eq:zero-one-CW}\notag
	u_{E,\phi} (z) = \sup \left\{ v(z) : v\in PSH(\bB), \, v_{|_E} \leq \phi, v\leq  \phi+ 1 \right\}.
\]
(Notice that these functions are different from the ones in \cite{CKZ} and \cite{Al-thesis}). It follows from the definitions and the assumption on $\phi$ that
\[\label{eq:re-ext-compare-c}
	u_{E} +\inf_E\phi \leq u_{E,\phi} \leq  2u_E + \sup_E \phi. 
\]
Note that this inequality is very similar to \eqref{eq:weight-vs-unweight} and Lemma~\ref{lem:weight-vs-unweight}-(a). Moreover, the extra constant 2 on the right hand side will not cause any harm in study the regularity of the subset $E$. 
Furthermore, we also have the monotonicity for  $E\subset F$ and $0\leq \phi \leq \psi \leq 1$, 
\[\label{eq:re-ext-monotonicity}
	u_{F,\phi} \leq u_{E,\phi} \leq u_{E, \psi}.
\]

We are ready to prove  the relation between the global and local extremal functions in a coordinate ball $\Om\subset X$  defined  in \eqref{eq:c-ball}.

\begin{lem} \label{lem:equivalence-notions}  Let $K\subset\subset \Om$ be a non-pluripolar compact set and $\wh \rho := \rho\circ \tau^{-1}$.  Then, there exist two positive constants $m, M$ depending  only on $K,\Om$ and $\rho$ such that 
$$
 m\,  u_{\tau(K),\wh\rho} (z) \leq (V_{K} +\rho) \circ \tau^{-1} (z) \leq M \, u_{\tau(K),\wh\rho} (z) \quad \text{on } \bB.
$$
\end{lem}
\begin{proof}  We follow closely the argument in \cite[Proposition~5.3.3]{Kl91}. Since $K$ is a  non-pluripolar, we have $M_\Om= \sup_\Om V_K^* <+\infty$. 
Let $w\in PSH(X,\om)$ such that $w\leq 0$ on $K$. 
Hence, $v: =(w +\rho)\circ \tau^{-1}$ is  plurisubharmornic in $\bB$, $v \leq \wh\rho$ on $\tau(K)$ and $ v  \leq M_\Om +1 =:M$. Since $\wh \rho \geq 0$, $$v \leq M u_{\tau(K),\wh\rho}$$
and therefore the second inequality follows.  Before continuing the proof of the first inequality let us state a useful observation (see also \cite[Corollary~1.7]{BBW11} and \cite[Theorem~2.7]{DMN} for different proofs).

\begin{cor}\label{cor:L-reg-smooth-domain} Let $D \subset\subset X$ be a smooth domain. Then, $V_{\ov D}$ is continuous on $X$.
\end{cor}
\begin{proof} Since $V_{\ov D}=0$ on $\ov D$, it is continuous in $D$. By Corollary~\ref{cor:semicontinuity} it remains to verify $V^*_{\ov D}(a)=0$ at  $a\in \d D$. Let us fix such a point $a$ and consider $\Om$ the local coordinate unit ball center at $a$ as in \eqref{eq:c-ball}. Set $K = \ov D \cap \tau^{-1}(\ov\bB(0,1/2))$. By monotonicity Proposition~\ref{prop:elementary}-(a) it  is enough to show that $V_{K}^*(a)=0$. To this end we note that $\tau(K) \subset \bar\bB(0,1/2)$ is a smooth domain near the origin $0=\tau(a)$. So, it is locally $L$-regular at $0$  by Example~\ref{expl:acc}. Hence, $u_{\tau(K)\cap \bar\bB(0,\veps)}^*(0) =0$ for every $\veps>0$ small by the inequality \eqref{eq:compare-C}. Next, using the second inequality above, the monotonicity \eqref{eq:re-ext-monotonicity} and \eqref{eq:re-ext-compare-c} we get
$$\begin{aligned}
	V_{K} \circ \tau^{-1}(z) +  \wh\rho(z) 
& \leq M u_{\tau(K),\wh\rho} (z)	\\
&\leq M u_{\tau(K)\cap \bar\bB(0,\veps),\wh\rho}  (z)\\
&\leq M [2 u_{\tau(K) \cap \bar\bB(0,\veps)} (z)+ \sup_{\bar\bB(0,\veps)} \wh\rho].
\end{aligned}
$$
Using the fact $\wh\rho (0)=0$ we get 
$
	V_K^*(a) \leq 2M u_{\tau(K)\cap \bar\bB(0,\veps)}^*(0) + Mc_3 \veps = Mc_3\veps,
$
where $c_3$ is the Lipschitz norm of $\wh\rho$ in $\bar\bB$. As $\veps >0$ arbitrary we have $V_K^*(a)=0$ and the proof of the corollary is completed.
\end{proof}

Let us complete  now the proof the first inequality in the lemma. By the assumption $2m :=\inf_{\d \Om} \rho >0$. Take $0<\veps<m$. For $\de>0$ denote $K_\de = \{ x\in X: {\rm dist}(x, K) \leq \de\}$ where ${\rm dist} (\cdot, K)$ is the distance induced from the metric $\om$. If $\de>0$ is small enough, then $K_\de\subset\subset \Om$ is a smooth domain. Hence, by Corollary~\ref{cor:L-reg-smooth-domain}  we have $V_{K_\de}$ is a continuous $\om$-psh function. Let $v \in PSH(\bB)$ be such that $v \leq \wh\rho$ on $\tau(K)$ and $v\leq \wh\rho+1$ in $\bB$. Define
$$
\wt v =
\begin{cases}
	\max\{ (m-\veps) v \circ \tau -\rho, V_{K_\de}\} \quad &\text{in } \Om,\\
	V_{K_\de} \quad&\text{in } X\setminus\Om.
\end{cases}
$$
Since $\limsup_{x\to \d\Om} [(m-\veps) v\circ \tau (x) - \rho(x)] \leq 2(m-\veps) - \inf_{\d \Om} \rho \leq -2\veps$, we have $(m-\veps) \,v \circ \tau -\rho \leq V_{K_\de}$ near $\d\Om$. Hence, 
 $\wt v \in PSH(X,\om)$ and $\wt v \leq 0$ on $K$. Hence, 
$$ (m-\veps) v \circ \tau -\rho \leq V_K \quad \text{ in }\Om.$$ 
By letting $\veps\to0$, we get the first inequality.
\end{proof}

For a general compact set $K\subset X$, we know from  Theorem~\ref{thm:characterization-hcp}-(a)  that $V_K$ is $\mu$-H\"older continuous at $a$ if and only if $V_{K\cap B(a,r)}$ is $\mu$-H\"older continuous for a closed coordinate ball $B(a,r)= \tau^{-1}(\bB(0,r))$ centered at $a$ and small radius $r>0$. Therefore, we may assume that $K$ is contained in this  holomorphic coordinate ball $(\Om,\tau)$ above and $K\subset B(a,1/2):= \tau^{-1}(\ov\bB(0,1/2))$. Clearly, for $0<r<1/2$
\[
	\tau(K\cap B(a,r)) = \tau(K) \cap \ov\bB(0,r) \subset \ov\bB (0,1/2).
\]
We now state the characterization of the local $\mu$-HCP.  
\begin{lem} $\tau(K)$ has local $\mu$-HCP of order $q$ at $\tau(a)=0$ if and only if $K$ has a uniform density in capacity  \eqref{eq:cap-density} at $a$ and 
$V_K$ is $\mu'$-H\"older continuous at $a$ for some $\mu'>0$.
\end{lem}

\begin{proof} Let us write  $E := \tau(K)$, $\ov\bB_r:= \ov\bB(0,r)$ and assume that $E$ has local $\mu$-HCP at $\tau(a)=0$, i.e.,
\[\label{eq:a-loc-hcp-1}
	L_{E \cap \ov\bB_r} (z) \leq \frac{C \de^\mu}{r^q}, \quad |z| \leq \de \leq \de_0, \quad 0< r \leq r_0.
\]
By the proof of Lemma~\ref{lem:cap-density} $E$ satisfies the property \eqref{eq:cap-density-sufficient-loc} at $\tau(a)=0$. Equivalently, $K$ satisfies \eqref{eq:cap-density} at $a$.
Lemma~\ref{lem:equivalence-notions} tells us that
\[\label{eq:equi-ext-functions-a}
	V_K \circ\tau^{-1}(z) \leq V_{K}\circ \tau^{-1} (z) + \wh \rho (z) \leq M u_{E, \wh \rho} (z).
\]
Furthermore, the monotonicity \eqref{eq:re-ext-compare-c} combined with \eqref{eq:compare-C} and \eqref{eq:a-loc-hcp-1}  implies for $0<r \leq r_0$
 $$\begin{aligned}
 	u_{E \cap \ov\bB_r, \wh \rho} (z) 
&\leq 	2u_{E \cap \ov\bB_r} (z) + \sup_{E \cap \ov\bB_r} \wh\rho \\
&\leq		\frac{2}{C_1} L_{E \cap \ov\bB_r} (z) + c_3 r \\
&\leq 	C \frac{\de^\mu}{r^q} + c_3 r.
 \end{aligned} $$
where $c_3$ is the Lipschitz norm of $\wh \rho$ on $\bB$. We can choose $r= \de^\eps$ with $\eps = \mu/(q+1)$ to conclude that  for $\mu' = \mu^2/(1+q)$,
$$
	u_{E,\wh \rho} (z) \leq C \de^{\mu'},\quad  |z| \leq \de \leq \de_0.
$$
It follows from this and \eqref{eq:equi-ext-functions-a} that
$$
	V_{K}  \circ \tau^{-1} (z) \leq M C \de^{\mu'} \quad\text{for } |z| \leq \de
$$
which is the H\"older continuity of $V_K$ at $a$. 

Conversely, assume now  $V_K$ is $\mu$-H\"older continuous and $K$ has a uniform density in capacity \eqref{eq:cap-density} at $a$. This implies from 
\eqref{eq:norm-loc} that for ${\rm dist} (x, a) \leq \de \leq \de_0$ and $0<r\leq r_0$,
$$
	V_{K\cap B(a,r)} (x) \leq  \frac{A}{\vka} \frac{\de^\mu}{r^{2+q}}.
$$
Since $\inf_{E\cap \ov\bB_r} \wh\rho=\wh\rho(0)=0$, we have for $|z| \leq \de \leq \de_0$,
$$
	\wh \rho (z) - \inf_{E\cap \ov\bB_r} \wh\rho \leq c_3 \de.
$$
Combining this with \eqref{eq:compare-C}, \eqref{eq:re-ext-compare-c} and Lemma~\ref{lem:equivalence-notions} we arrive at
$$\begin{aligned}	
	L_{E \cap \ov\bB_r} &\leq C_2 u_{E\cap \ov\bB_r} \leq C_2 \left(u_{E\cap \ov\bB_r,\wh\rho} - \inf_{E\cap \ov\bB_r} \wh\rho \right) \\
&\leq \frac{C_2}{m} V_{K\cap B(a,r)}\circ \tau^{-1} + \frac{C_2c_3}{m}\de.
\end{aligned}
$$
Therefore, for $0<r\leq r_0$ and $|z|\leq \de \leq \de_0$,
$$
	L_{E \cap \ov\bB_r}(z) \leq  \frac{A_1 \de^\mu}{r^{2+q}} + c_4 \de,
$$
where $A_1, c_4$ are  uniform constants. This is the desired inequality for $L_{E\cap\ov\bB_r}$.
\end{proof}

\end{document}